\documentclass{article}
\usepackage{subfig}
\usepackage[section] {placeins}
\usepackage{amsthm}
\usepackage{amssymb,amsmath}
\usepackage{anysize}
\usepackage{psfrag}

\usepackage{epsfig}
\usepackage{psfrag}
\usepackage{lscape}

\newtheorem{thm}{Theorem}[section]
\newtheorem{def.}{Definition}[section]

\newtheorem{prop}{Proposition}[section]
\newtheorem{cor}{Corollary}[section]

\newtheorem{lem}{Lemma}[section]

\numberwithin{table}{section}
\begin{document}

\title{The Delunification Process and Minimal Diagrams}
\author{Slavik Jablan\\
        The Mathematical Institute\\
        Knez Mihailova 36\\
        P.O. Box 367, 11001, Belgrade\\
        Serbia\\
        \texttt{sjablan@gmail.com}\\
        and\\
        Louis H. Kauffman\\
        Department of Mathematics, Statistics and Computer Science\\
        University of Illinois at Chicago\\
        851 S. Morgan St., Chicago IL 60607-7045\\
        USA\\
        \texttt{kauffman@uic.edu}\\
        and\\
        Pedro Lopes\\
        Center for Mathematical Analysis, Geometry and Dynamical Systems\\
        Department of Mathematics\\
        Instituto Superior T\'ecnico, Universidade de Lisboa\\
        1049-001 Lisboa\\
        Portugal\\
        \texttt{pelopes@math.tecnico.ulisboa.pt}\\
}
\date{June 09, 2014}
\maketitle

\begin{abstract}
A link diagram is said to be lune-free if, when viewed as  a 4-regular plane graph it does not have multiple edges between any pair of nodes. We prove that any colored link diagram is equivalent to a colored lune-free diagram with the same number of colors. Thus any colored link diagram with a  minimum number of colors (known as a minimal diagram) is equivalent to a colored lune-free diagram with that same number of colors. We call the passage from a link diagram to an equivalent lune-free diagram its delunification process.

We then  introduce a notion of grey sets in order to obtain  higher lower bounds for minimum number of colors. We calculate these higher lower bounds for a number of prime moduli with the help of computer programs.

For each number of crossings through 16, we list the lune-free diagrams and we color them. If the number of colors equals the corresponding higher lower bound we know we have a minimum number of colors. We also introduce and list the lune-free crossing number of a link i.e., the minimum number of crossings needed for a lune-free diagram of this link, and other related link invariants.
\end{abstract}

\bigbreak


\bigbreak


\bigbreak

\section{Introduction.}

\noindent

In this article we consider Fox colorings of link diagrams \cite{CFox} and their minimality properties.
A Fox coloring is a labeling of the arcs of the link diagram with elements of the integers modulo $m$ for an appropriate modulus $m.$ Such colorings can be regarded as labelings in a quandle with operation
$a*b = 2b -a,$ and are related to properties of the classical double-branched covering space with branch set that knot or link. There are many minimality questions about such colorings, since it is often the case that, for a given diagram, not all $m$ elements of the modular arithmetic are needed to color that diagram. Thus we consider the minimum coloring number of a knot or link to be the least number of colors that suffice to produce a non-trivial coloring (among all possible diagrams for the link) in a given modulus $m$, notation, $mincol_m\, L$, for a link or knot $L$.
\bigbreak

A link diagram is said to be {\it lune-free} if it does not have any two-sided regions
\cite{EliahouHararyKauffman}.
An equivalent term for lune-free is to say that the underlying flat diagram is a {\it Conway polyhedron}.
This terminology originated with J. H. Conway's paper \cite{Conway} in which he used a few basic polyhedra and insertions of rational diagrams in them, to produce the complete tables of knots up through ten crossings. Since that time it has been recognized that lune-free diagrams form a core structure for the class of all link diagrams.
\bigbreak

In this article we prove that if a link, $L$, admits a non-trivial coloring modulo a positive integer $m$, then there is a lune-free diagram of this link which supports such a non-trivial coloring using the minimum number of colors, $mincol_m L$. This is the consequence of the Main Lemma that we prove below which shows that  a colored lune is eliminated by a given finite sequence of colored Reidemeister moves which preserves the number of colors. Since for any link, $L$, admitting non-trivial colorings mod $m$, there is a diagram supporting a non-trivial coloring using the $mincol_m L$ colors, then if this diagram is not a lune-free diagram, we can use the finite sequence of moves described in the Main Lemma to obtain a lune-free diagram of $L$ using $mincol_m L$. It entails another interesting result: any link can be represented by a lune-free diagram.  We call this passage from a link diagram to an equivalent lune-free diagram the \emph{delunification process}.

\begin{thm}\label{thm:main}
Let $m$ be a positive integer greater than $1$. Let $L$ be a link admitting non-trivial $m$-colorings. There is a lune-free diagram of $L$ which supports a non-trivial $m$-coloring using the least number of colors, $mincol_m L$.
\end{thm}

The proof of Theorem \ref{thm:main} is an immediate consequence of the Main Lemma (Lemma \ref{lemma: main}) which is proved below in Section 2.

The relevance of the existence of lune-free diagrams supporting minimal colorings is that we may search for minimum number of colors in the smaller subclass of lune-free diagrams. Also, due to their rigidity, it is easier to list the lune-free diagrams of a given number of crossings than to list all the diagrams for this number of crossings.
\bigbreak
The study of minimum number of colors initiated with the article \cite{Frank}, and was carried on in a number of other articles where the authors try to obtain estimates for the minimum number of colors for links of specific families (\cite{kl, klgame,  lm1}), or try to prove that links admitting non-trivial colorings on a given modulus all have the same minimum number of colors. The latter statement is in fact the case for moduli $2$, $3$, $5$, and $7$ (\cite{satoh, Oshiro, Saito, lm}). Moreover, for each of these moduli, there is a specific set of colors, whose cardinality is the minimum number of colors for the modulus at stake, with which such a minimal coloring can be assembled. But in \cite{lopesp11} it is proved that at $p=11$ this pattern breaks down. Specifically, knots $6_2$ and $7_2$ both admitting non-trivial $11$-colorings, require distinct sets of colors in order to assemble minimal $11$-colorings. Although the cardinality of these minimal sets of colors is $5$ for both $6_2$ and for $7_2$, this raises the following question. Does the minimum number of colors depend exclusively on the modulus at stake? That is to say, are there distinct  knots (or links), $L$ and $L'$, both admitting non-trivial $p$-colorings but such that $mincol_p\, L \neq mincol_p\, L'$? These questions led us to trying to determine minimum number of colors for moduli higher than $7$ which in turn gave rise to the current article.
\bigbreak

{\bf Remark} We warn the reader that any link in this article is considered to have non-null determinant. As a matter of fact, links with null determinant admit non-trivial colorings on any modulus.  We therefore think of them as forming a special class of links which we plan on addressing in a separate article.

\bigbreak
We remark that most of the results of this article and the new definitions go over to the class of virtual knots \cite{VKT}. Coloring is defined for virtual knots in the same way as we have defined it in this article (by a relation at each classical crossing). Virtual crossings do not entail an extra coloring relation. We define a {\it lune} in a virtual diagram to be a region in that diagram with two sides, whose crossings are {\it both} classical. Then it is clear that our methods for
delunification apply for virtual diagrams, since they use local modifications that are not affected by the
presence of virtual crossings. Examples and consequences of these remarks for virtual knot theory will be the subject of a separate paper.
\bigbreak

The article is organized as follows. In Section \ref{sect:delun} we show how to obtain a lune-free diagram from a diagram equipped with a non-trivial coloring, while preserving the number of colors. Colors apart, this is the delunification process of a diagram. We also explore different ways of delunifying diagrams and estimate the excess of crossings that each one brings about. This in turn leads us to defining three new notions of crossing numbers. In Section \ref{sect:grey} we introduce the notion of grey sets in order to obtain higher lower bounds for the minimum number of colors and calculate these higher lower bounds for prime moduli through $43$. In Section \ref{sect:algo} we discuss the algorithms employed in the listing of the lune-free diagrams and present tables with the values obtained for the minimum number of colors, and for the distinct minimum number of crossings.

\bigbreak
\section{The delunification process and minimal diagrams.}\label{sect:delun}

\noindent

We start by defining a few notions which will  simplify the statements of our results.

\begin{def.}\label{def:minimal coloring}$[\emph{$m$-Minimal Coloring of $L$}]$
Let $m$ be a positive integer, let $L$ be a link admitting non-trivial $m$-colorings. An \emph{$m$-Minimal Coloring of $L$} is a diagram of $L$ equipped with a non-trivial $m$-coloring using $mincol_m L$ colors. $(m$ and/or $L$ will be dropped from $m$-Minimal Coloring of $L$ whenever $m$ and/or $L$ are clear from context.$)$
\end{def.}

\begin{def.}\label{def:colReid}$[\emph{$m$-Colored Reidemeister Moves}]$
Let  $D$ be a link diagram equipped with a coloring over a given modulus $m$. An \emph{$m$-Colored Reidemeister Move} on $D$ is a Reidemeister move performed on $D$ along with the unique reassignment of colors to the arcs brought about by the Reidemeister move such that the new diagram is also equipped with a coloring mod $m$. This new coloring coincides with the former coloring in the arcs of the diagram not affected by the Reidemeister move $($\cite{pLopesqft1}$)$. $(m$ will be dropped from $m$-Colored Reidemeister Move whenever it is clear from context which $m$ is at issue.$)$
\end{def.}

\begin{def.}\label{def:maxtassel}$[\emph{$k$-Tassel; Maximal Tassel; Sub-Tassel; Isolated Lune.}]$
Let  $D$ be a link diagram.

A \emph{$k$-Tassel} is a portion of $D$ which is isotopic on the plane to $\sigma_1^k$ where $k$ is a non-null integer and $\sigma_1$ is a generator of the braid group corresponding to the second strand going over the first strand. $($The $k$- will be dropped whenever it is not meaningful or it is clear from context.$)$

A \emph{Maximal Tassel} is a Tassel which is not part of a larger Tassel in the diagram under consideration.

A \emph{Sub-Tassel} or \emph{Non-Maximal Tassel} is a Tassel which is a proper part of a larger Tassel in the diagram under consideration.

An \emph{Isolated Lune} in a diagram $D$ is a maximal $2$-tassel of $D$.

Lunes in a $($sub-$)$tassel will be referred to as \emph{consecutive lunes}.

$($See Figure \ref{fig:defmaxisononmax} for illustrative examples.$)$
\end{def.}

\begin{figure}[!ht]
    \psfrag{iso}{\huge\text{Isolated Lune}}
    \psfrag{max}{\huge\text{(Non-Isolated Lune) Maximal Tassel}}
    \psfrag{nonmax}{\huge\text{(Non-maximal) Tassel}}
    \psfrag{2b-a1}{\Large$2b-a$}
    \psfrag{3b-2a}{\huge$3b-2a$}
    \psfrag{3b-2a1}{\Large$3b-2a$}
    \psfrag{4b-3a}{\large$4b-3a$}
    \centerline{\scalebox{.47}{\includegraphics{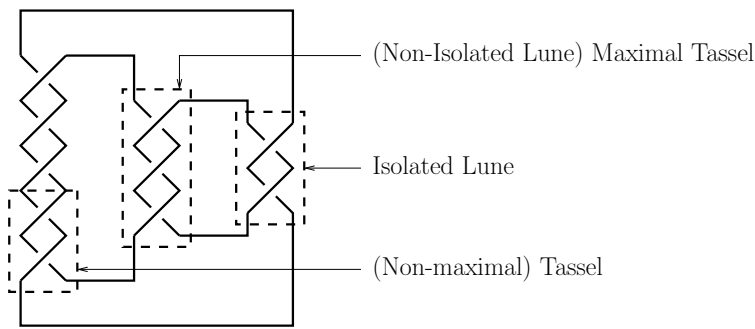}}}
    \caption{Illustrative Examples of Tassels (maximal and otherwise) and of Isolated Lune.}\label{fig:defmaxisononmax}
\end{figure}

\begin{lem}\label{lemma: main}$[\bf{Main\;  Lemma}]$ There is a sequence of colored Reidemeister moves that eliminates lunes in any colored link diagram while not creating new lunes and preserving the number of colors. This sequence of moves increases the number of crossings by $8$, per lune i.e., the neighborhood of the lune containing $2$ crossings, will contain, after this sequence, $2+8=10$ crossings.
\end{lem}
\begin{proof} See Figure \ref{fig:mainl}.
\begin{figure}[!ht]
    \psfrag{a}{\huge$a$}
    \psfrag{b}{\huge$b$}
    \psfrag{2b-a}{\huge$2b-a$}
    \psfrag{2b-a1}{\Large$2b-a$}
    \psfrag{3b-2a}{\huge$3b-2a$}
    \psfrag{3b-2a1}{\Large$3b-2a$}
    \centerline{\scalebox{.47}{\includegraphics{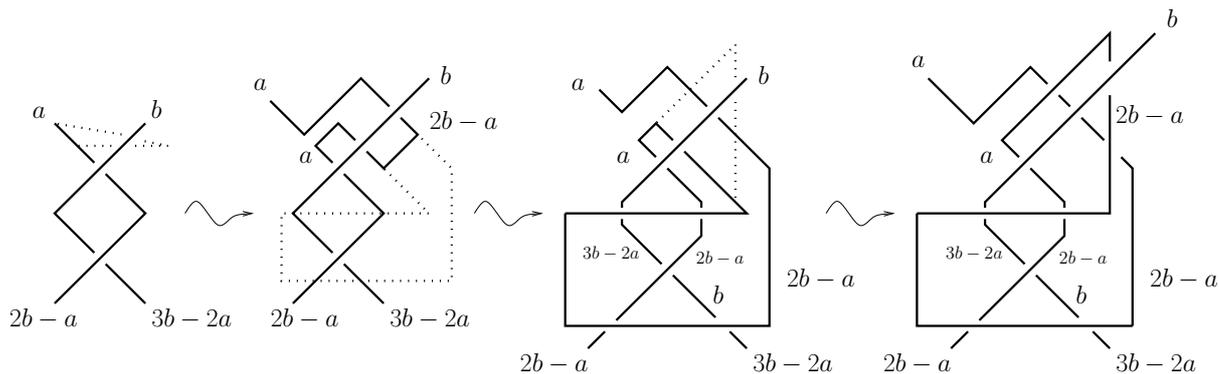}}}
    \caption{The Main Lemma: a sequence of colored Reidemeister moves which eliminates a lune while not creating new lunes and preserving the number of colors. The dotted lines indicate the move that realizes the passage to the diagram on the right.}\label{fig:mainl}
\end{figure}
\end{proof}

\begin{cor}\label{cor:lunefreediag} Any link can be represented by a lune-free diagram.
\end{cor}
\begin{proof} It is implicit in the proof of Lemma \ref{lemma: main} above but here is another sequence of Reidemeister moves which also eliminates lunes. This sequence, however, does not necessarily preserve the number of colors but involve a smaller increase in the number of crossings - see Figure \ref{fig:cormainl}.
\end{proof}
\begin{figure}[!ht]
    \psfrag{a}{\huge$a$}
    \psfrag{b}{\huge$b$}
    \psfrag{2a-b}{\huge$2a-b$}
    \psfrag{2b-a}{\huge$2b-a$}
    \psfrag{2b-a1}{\Large$2b-a$}
    \psfrag{3b-2a}{\huge$3b-2a$}
    \psfrag{3b-2a1}{\Large$3b-2a$}
    \psfrag{4b-3a}{\huge$4b-3a$}
    \centerline{\scalebox{.47}{\includegraphics{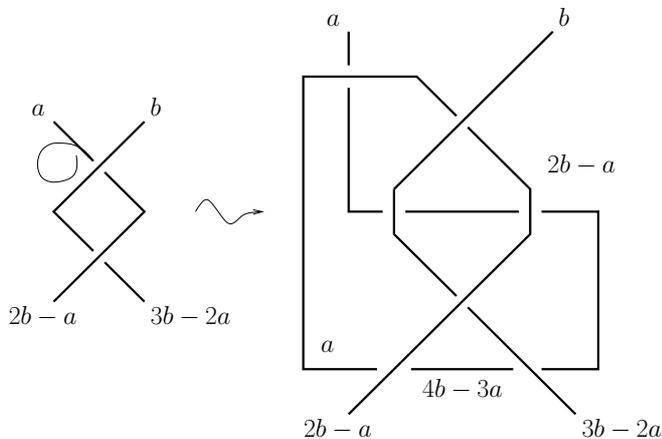}}}
    \caption{The Corollary to the Main Lemma: a sequence of colored Reidemeister moves which eliminates a lune. No more lunes are created and only $5$ extra crossings are produced. On the other hand one color is (locally) added. If this color is already present in another part of the diagram then no color is globally added by way of this sequence of moves. For instance (referring to diagram on the left), if arc colored $3b-2a$ (bottom right) crosses over arc colored $2b-a$ (bottom left) in a part of the diagram not displayed here, then no extra color is globally added - this argument will be used below in the proof of Lemma \ref{lemma:aux}.}\label{fig:cormainl}
\end{figure}

\begin{lem}\label{lemma:aux}$[\emph{Auxiliary Lemma}]$ Figures \ref{fig:cormainlv4}, \ref{fig:cormainlv5}, \ref{fig:cormainlv5+}, and \ref{fig:cormainlv6} indicate the sequences of colored Reidemeister moves which eliminate the consecutive lunes in colored maximal $3$-, $4$-, $5$-, and $6$-tassels (respectively)  without increasing the number of colors, and with less increase in the number of crossings than via the technique described in the Main Lemma $($Lemma \ref{lemma: main}$)$.
\end{lem}
\begin{proof} See Figures \ref{fig:cormainlv4}, \ref{fig:cormainlv5}, \ref{fig:cormainlv5+}, and \ref{fig:cormainlv6}.
\end{proof}

\begin{figure}[!ht]
    \psfrag{a}{\huge$a$}
    \psfrag{b}{\huge$b$}
    \psfrag{2a-b}{\huge$2a-b$}
    \psfrag{2b-a}{\huge$2b-a$}
    \psfrag{2b-a1}{\Large$2b-a$}
    \psfrag{3b-2a}{\huge$3b-2a$}
    \psfrag{3b-2a1}{\Large$3b-2a$}
    \psfrag{4b-3a}{\huge$4b-3a$}
    \centerline{\scalebox{.47}{\includegraphics{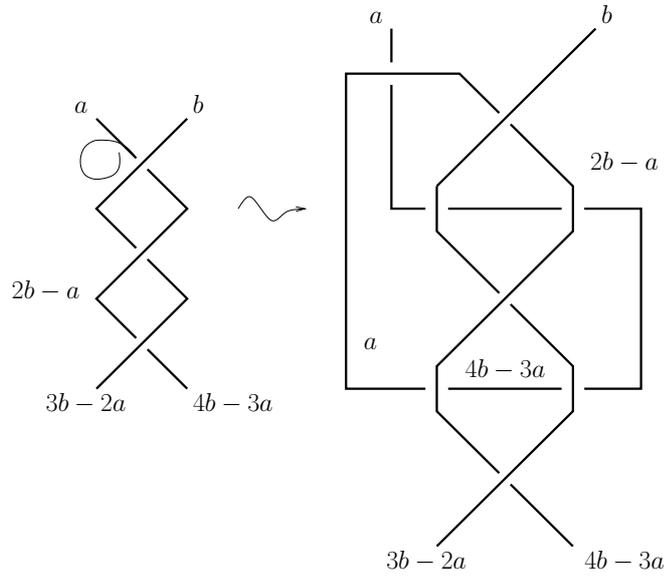}}}
    \caption{A sequence of Reidemeister moves which eliminates the lunes in a colored maximal $3$-tassel, preserving the number of colors and increasing the number of crossings by $5$.}\label{fig:cormainlv4}
\end{figure}
\begin{figure}[!ht]
    \psfrag{a}{\huge$a$}
    \psfrag{b}{\huge$b$}
    \psfrag{2a-b}{\huge$2a-b$}
    \psfrag{2b-a}{\huge$2b-a$}
    \psfrag{2b-a1}{\Large$2b-a$}
    \psfrag{3b-2a}{\huge$3b-2a$}
    \psfrag{3b-2a1}{\Large$3b-2a$}
    \psfrag{4b-3a}{\huge$4b-3a$}
    \psfrag{5b-4a}{\huge$5b-4a$}
    \centerline{\scalebox{.47}{\includegraphics{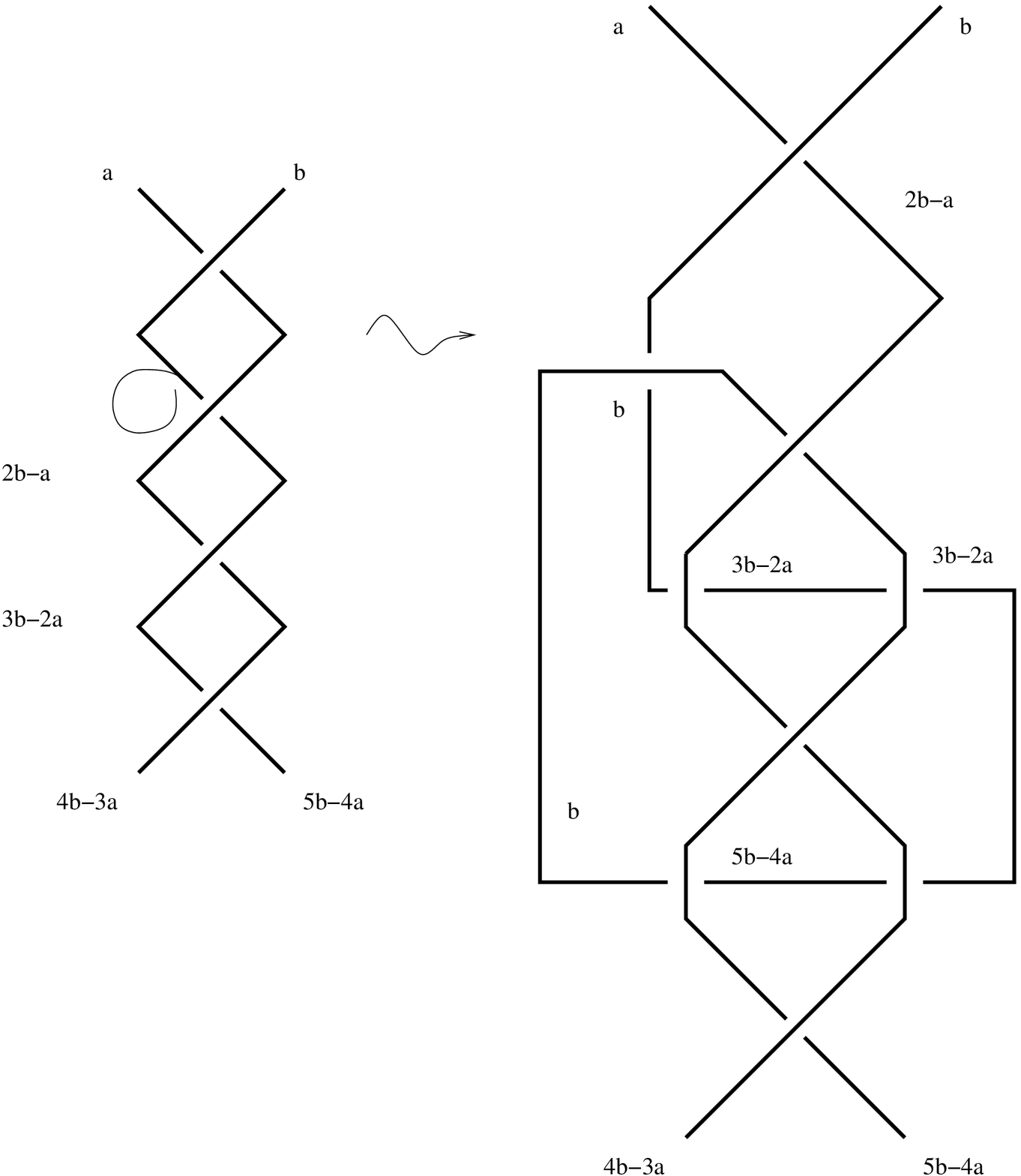}}}
    \caption{A sequence of Reidemeister moves which eliminates the lunes in a colored maximal $4$-tassel, preserving the number of colors and increasing the number of crossings by $5$.}\label{fig:cormainlv5}
\end{figure}
\begin{figure}[!ht]
    \psfrag{a}{\huge$a$}
    \psfrag{b}{\huge$b$}
    \psfrag{2a-b}{\huge$2a-b$}
    \psfrag{2b-a}{\huge$2b-a$}
    \psfrag{2b-a1}{\Large$2b-a$}
    \psfrag{3b-2a}{\huge$3b-2a$}
    \psfrag{3b-2a1}{\Large$3b-2a$}
    \psfrag{3b-2a1}{\Large$3b-2a$}
    \psfrag{4b-3a}{\huge$4b-3a$}
    \psfrag{4b-3a1}{\Large$4b-3a$}
    \psfrag{5b-4a}{\huge$5b-4a$}
    \psfrag{5b-4a1}{\Large$5b-4a$}
    \psfrag{6b-5a}{\huge$6b-5a$}
    \psfrag{6b-5a1}{\Large$6b-5a$}
    \psfrag{7b-6a1}{\Large$7b-6a$}
    \psfrag{7b-6a}{\huge$7b-6a$}
    \centerline{\scalebox{.47}{\includegraphics{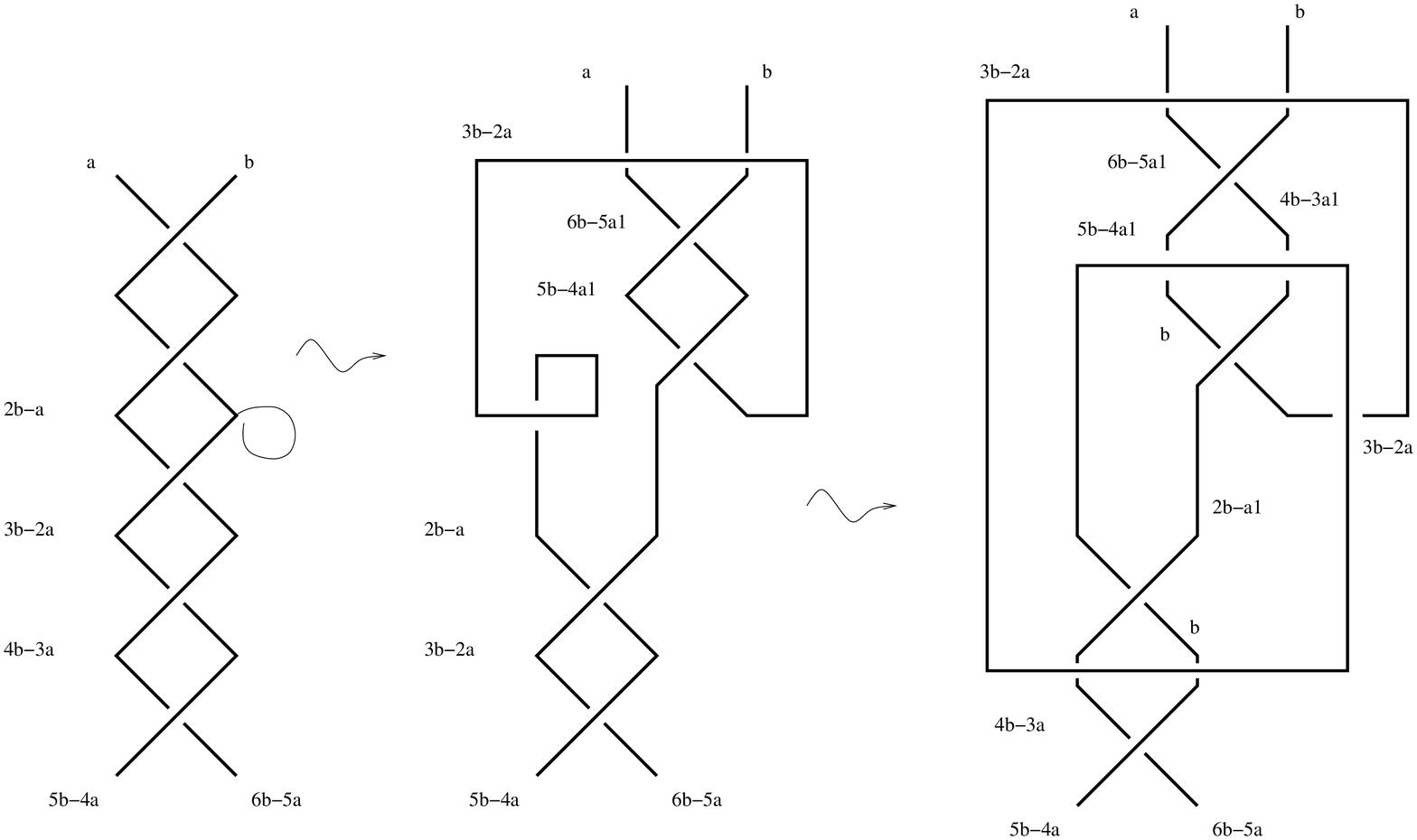}}}
    \caption{A sequence of Reidemeister moves which eliminates the lunes in a colored maximal $5$-tassel, preserving the number of colors and increasing the number of crossings by $6$.}\label{fig:cormainlv5+}
\end{figure}
\begin{figure}[!ht]
    \psfrag{a}{\huge$a$}
    \psfrag{b}{\huge$b$}
    \psfrag{2a-b}{\huge$2a-b$}
    \psfrag{2b-a}{\huge$2b-a$}
    \psfrag{2b-a1}{\Large$2b-a$}
    \psfrag{3b-2a}{\huge$3b-2a$}
    \psfrag{3b-2a1}{\Large$3b-2a$}
    \psfrag{3b-2a1}{\Large$3b-2a$}
    \psfrag{4b-3a}{\huge$4b-3a$}
    \psfrag{5b-4a}{\huge$5b-4a$}
    \psfrag{5b-4a1}{\Large$5b-4a$}
    \psfrag{6b-5a}{\huge$6b-5a$}
    \psfrag{6b-5a1}{\Large$6b-5a$}
    \psfrag{7b-6a1}{\Large$7b-6a$}
    \psfrag{7b-6a}{\huge$7b-6a$}
    \centerline{\scalebox{.47}{\includegraphics{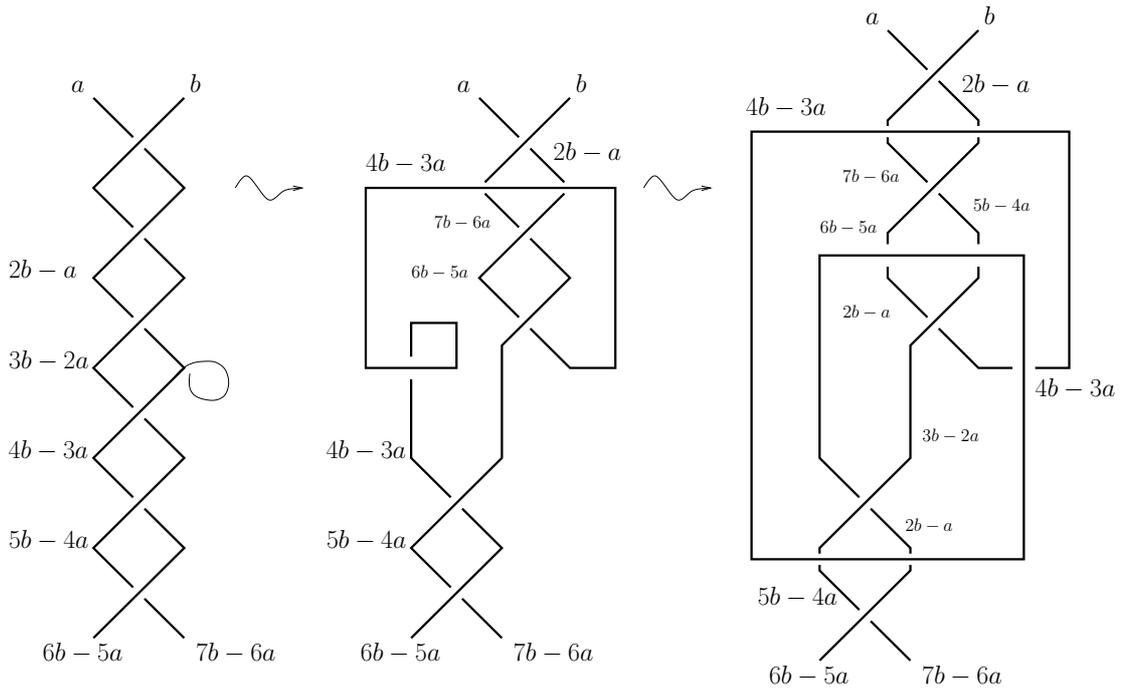}}}
    \caption{A sequence of Reidemeister moves which eliminates the lunes in a colored maximal $6$-tassel, preserving the number of colors and increasing the number of crossings by $6$.}\label{fig:cormainlv6}
\end{figure}

At this point the following remarks seem to be in order. A maximal $k$-tassel involves $k-1$ consecutive lunes.  Figure \ref{fig:cormainlv5} describes a technique for eliminating $3$ consecutive lunes; Figure \ref{fig:cormainlv4} is an application of this technique to eliminating $2$ consecutive lunes. Figure \ref{fig:cormainlv6} describes a technique for eliminating $5$ consecutive lunes; Figure \ref{fig:cormainlv5+} is an application of this technique to eliminating $4$ consecutive lunes.

\begin{cor}\label{cor:count}$[\emph{Counting extra crossings}]$
Consider a maximal $n$-tassel in a colored diagram. The application of Auxiliary Lemma \ref{lemma:aux} in the elimination of the colored lunes of this tassel brings about the following increase in the number of crossings $(n\geq 6)$.
\begin{itemize}
\item $6k$ \qquad\quad \; if $n=6k$;
\item $6k+4$ \qquad if $n=6k+1$, or $n=6k+2$;
\item $6k+5$ \qquad if $n=6k+3$, or $n=6k+4$;
\item $6k+6$ \qquad if $n=6k+5$.
\end{itemize}
\end{cor}
\begin{proof}
$n=6k+k'$ for positive $k\geq 1$ and $0 \leq k' < 6$.

If $n=6k$  then each sub-tassel $\sigma_1^6$ is treated as in Figure \ref{fig:cormainlv6}.

If $n=6k+1$ treat the $\sigma_1^{6(k-1)}$ part of the $n$-tassel as in the preceding case (obtaining from it $6(k-1)$ extra crossings), and the $\sigma_1^{6+1}=\sigma_1^3\cdot \sigma_1^4$ part of it as in Figure \ref{fig:cormainl} (for the $\sigma_1^3$ part of it and obtaining another $5$ extra crossings) and as in Figure \ref{fig:cormainlv4} (for the $\sigma_1^4$ part of it and obtaining another $5$ extra crossings). Analogously for $n=6k+2=6(k-1)+4+4$.

If $n=6k+3$ (respect., $n=6k+4$) then the $\sigma_1^{6k}$ is treated as before giving rise to $6k$ extra crossings. The $\sigma_1^3$ (respect., $\sigma_1^4$) part of it is treated as in Figure \ref{fig:cormainl} (respect., as in Figure \ref{fig:cormainlv4}) giving rise to another $5$ extra crossings.

If $n=6k+5$ the we write $\sigma_1^{6k+5}=\sigma_1^{6k}\cdot \sigma_1^5$ and reason analogously. This concludes the proof.
\end{proof}

\subsection{Review of the Teneva Game (\cite{klgame}).}\label{subsect:tenevagame}
\noindent
Proposition \ref{prop:Tenevaremove} below presents a different delunification of a maximal tassel. It is based on a procedure for breaking down braid-closed tassels i.e., torus links of type $(2, n)$, in order to estimate the minimum number of colors they admit modulo their determinant (which is $n$, the number of crossings of the tassel at stake). This procedure is called the Teneva Game (\cite{klgame}) and the effect of one iteration of it is called a Teneva transformation. Roughly speaking, a Teneva transformation is a finite sequence of colored Reidemeister moves which splits $\sigma_1^{2k+1}$ (respect., $\sigma_1^{2k}$) into two $\sigma_1^{k}$'s (respect., into $\sigma_1^{k}$ and $\sigma_1^{k-1}$) and reduces the number of colors to roughly half the original number of colors. A Teneva transformation is illustrated in Figure \ref{fig:tenevagame} for odd $n=5$ (left-hand side) and for even $n=6$ (right-hand side).

\begin{figure}[!ht]
   \psfrag{s2k+1}{\huge$\sigma_{1}^{2k+1}$}
   \psfrag{s2k}{\huge$\sigma_{1}^{2k}$}
   \psfrag{sk-1}{\huge$\sigma_{1}^{k-1}$}
   \psfrag{sk}{\huge$\sigma_{1}^{k}$}
   \psfrag{3b-2a}{\huge$3b-2a$}
   \psfrag{3b-2a1}{\Large$3b-2a$}
   \psfrag{4b-3a}{\large$4b-3a$}
    \centerline{\scalebox{.47}{\includegraphics{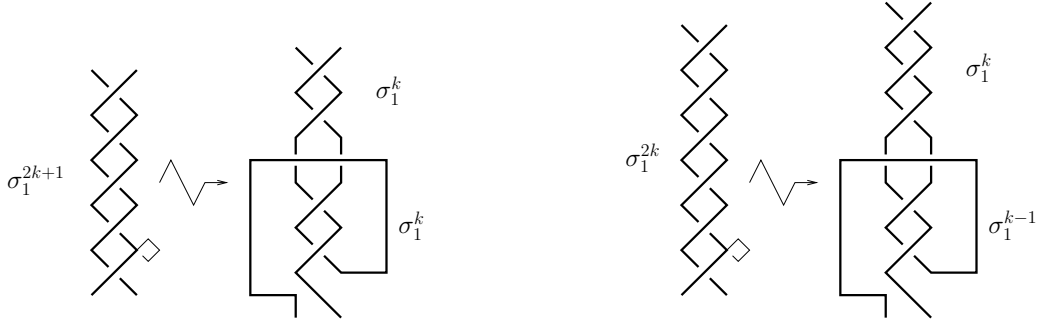}}}
    \caption{$\sigma_{1}^{n}$ and how it looks like after a Teneva transformation. The pair on the left illustrates the odd $n$ instance. Along with each Teneva transformation, the effect on the size of the tassels is displayed: $2k+1 \rightarrow k + k$ whereas $2k \rightarrow k+ (k-1)$.}\label{fig:tenevagame}
\end{figure}

If,  after the Teneva transformation has been performed, the remaining $\sigma_{1}^{k}$'s  exhibit $k > 4$, a new Teneva transformation can be performed on each $\sigma_{1}^{k}$. This is the Teneva Game. The game ends when the $\sigma_{1}^k$'s resulting from a Teneva transformation exhibit $k\in \{ 2, 3, 4  \}$ (\cite{klgame}).

We remark that although the Teneva Game was conceived for a task which a priori had nothing to do with delunification, it helps in the delunification process of a tassel. At the end of the Teneva Game we only have maximal tassels of the sort $\sigma_1^{\pm 2}$, $\sigma_1^{\pm 3}$, or $\sigma_1^{\pm 4}$, which can then be dealt with with the methods described in the Main Lemma \ref{lemma: main} or in Lemma \ref{lemma:aux}. It thus seemed relevant to ascertain which of these two methods brings about more crossings. The two methods we refer to are $(i)$ the systematic use of the methods described in Corollary \ref{cor:count},  and $(ii)$ the use of the Teneva Game first and eventually the use of the methods in the Main Lemma \ref{lemma: main} or in Lemma \ref{lemma:aux} to deal with the remaining   $\sigma_1^{\pm 2}$'s , $\sigma_1^{\pm 3}$'s , or $\sigma_1^{\pm 4}$'s.

\bigbreak

In the set-up of the Teneva Game, the following definitions are relevant.

\begin{def.}\label{def:lowerhalf}$\cite{klgame}$$[\emph{Lower Half of a positive integer; Sequence of lower halves of a positive integer}]$

For any positive odd integer $2k+1$ we set
\[
lh(2k+1):=k
\]
and call it the \emph{Lower Half} of $2k+1$. The lower half of a positive even integer coincides with the ordinary half.
\bigbreak
Given a positive integer $n\geq 5$, we define its \emph{Sequence of Lower Halves}, notation $LH(n)$, to be the sequence of iterates of the map $lh(\dots )$ on $n$. Its last term is the first iterate to lie in the set $\{  2, 3, 4 \}$. For instance, $LH(5)=(2)$, $LH(6)=(3)$, $LH(7)=(3)$, $LH(8)=4$, $LH(9)=4$, $LH(10)=(5, 2)$, $LH(11)=(5, 2)$, etc (more such calculations at the end of Section \ref{sect:grey}). 
\bigbreak

Furthermore, given a positive integer $n$, we define $l_n$ to be the number of entries in $LH(n)$, and call it \emph{the length of the sequence of lower halves of $n$}; and we define $t_n$ to be the last entry of $LH(n)$ $($necessarily $t_n \in \{ 2, 3, 4 \})$, and call it \emph{the tail of the sequence of lower halves of $n$}.
\end{def.}

We now state the main result of the Teneva Game which will be useful below in Section \ref{sect:grey}.
\begin{thm}\label{thm:thenevagame}$\cite{klgame}$ For any prime $p > 7$,
\[
mincol_p\, T(2, p) \leq t_p + 2l_p -1
\]
with $l_p$, and $t_p$ as defined in Definition \ref{def:lowerhalf}.
\end{thm}

\bigbreak

Another result from \cite{klgame} which will be useful below when comparing the two delunification processes is the following.
\begin{prop}\label{prop:2adicexp}$\cite{klgame}$
Given a positive integer $n$, consider its $2$-adic expansion:
\[
n= 2^{e_1} + 2^{e_2} + \cdots + 2^{e_{N_n}} + 1 \qquad \qquad \text{ with } e_1 > e_2 > \cdots > e_{N_n} \geq 0
\]

Then
\begin{enumerate}
\item If $e_1 - e_2 = 1$ then
\[
l_n = e_1-1 \qquad \qquad \text{ and } \qquad \qquad t_n=3
\]

\item If $e_1 - e_2 = 2$ then
\[
l_n = e_1-1 \qquad \qquad  \text{ and }  \qquad \qquad  t_n=2
\]

\item If $e_1 - e_2 > 2$ then
\[
l_n = e_1-2 \qquad \qquad  \text{ and }  \qquad \qquad t_n=4
\]
\end{enumerate}
with $l_n$ and $t_n$ as defined in Definition \ref{def:lowerhalf}.
\end{prop}

\bigbreak
\subsection{An alternative delunification process based on the Teneva Game.}\label{subsect:deluneficationtenevagame}
\noindent

We now  give  an alternative delunification of a maximal tassel  along with an estimate of the extra crossings it brings about.
\begin{prop}\label{prop:Tenevaremove}$[\emph{Counting extra crossings - 2nd approach}]$
Let $m$ be a positive integer, let $L$ be a link admitting non-trivial $m$-colorings. Let $D_0$ stand for an $m$-minimal coloring of $L$. If $D_0$ is not a lune-free diagram, we do the following on each maximal tassel in $D_0$. We  apply the Teneva Game $($\cite{klgame}$)$ to the maximal tassel at issue. At the end of the Teneva Game we obtain $\sigma_1^ i$'s $($where  $i=\pm 2$, or  $\pm 3$, or $\pm 4 )$. We then treat each of these $\sigma_1^i$'s as in the Main Lemma \ref{lemma: main} or as in Lemma \ref{lemma:aux}. The upper bounds on the increase in the number of the crossings at the end of the process are:
\begin{equation*}
\begin{cases}
6\cdot 2^l -1 \qquad & \text{ if there are no $\sigma_1^{\pm 2}$ at the end of the game} \\
9\cdot 2^l -1  \qquad & \text{ otherwise}
\end{cases}
\end{equation*}
where $l$ is the length of the Lower Half Sequence of the number of crossings of the maximal tassel at issue $($\ref{def:lowerhalf}$)$. \end{prop}
\begin{proof}A Teneva transformation consists of one type I Reidemeister move, producing one crossing, followed by a number of type III Reidemeister moves, see Figure \ref{fig:tenevagame}. Thus, each Teneva transformation introduces, per se, one extra crossing, associated with the performance of the type I Reidemeister move.

We assume we are dealing with a maximal $n$-tassel for which we set $l=l_n$.

When this tassel undergoes the Teneva Game it picks up $2^{k-1}$ crossings at the $k-th$  sequence of Teneva transformations (one crossing per type I Reidemeister move performed at this $k-th$ step of the Teneva Game). The overall increase in the number of crossings is then
\[
\sum_{k=0}^{l-1}2^k = 2^l-1
\]
Then we have to address the maximal tassels left over in the last step of the Teneva game. These are of the sort $\sigma_1^{\pm 2}$, $\sigma_1^{\pm 3}$, and/or $\sigma_1^{\pm 4}$. There are $2^l$ maximal tassels in this last step, so if there are no $\sigma_1^{\pm 2}$'s at this last step, the increase in the number of crossings is $5\cdot 2^l$, leaning on Lemma \ref{lemma:aux}. Otherwise the upper bound in the increase of crossings at this last step is $8\cdot 2^l$. Adding to these the number of crossings introduced before the end of the game, $2^l-1$, we obtain the results in the statement. The proof is concluded.
\end{proof}

\subsection{Comparing the two delunification processes.}\label{subsect:compare}

\noindent

We now compare the approaches described by Corollary \ref{cor:count} and Proposition \ref{prop:Tenevaremove} to see which one of them increases the least the number of crossings. Since each one of them resorts to a different parameter to express the results we have to express these parameters in terms of a common one. Specifically, if we are dealing with a maximal $n$-tassel with $n=6k+k'$, then Corollary \ref{cor:count} expresses results in terms of $k$ whereas Proposition \ref{prop:Tenevaremove} does this in terms of $l=l_{6k+k'}$. We will then write down the dyadic expansion of $k$ and from it extract information on the dyadic expansion of $6k+k'$ and on $l_{6k+k'}$. This will allow us to rewrite the results of Corollary \ref{cor:count} and Proposition \ref{prop:Tenevaremove} in terms of $k$.

\begin{prop}\label{prop:compare}
The approach described in Corollary \ref{cor:count} gives rise to less crossings than the approach described in Proposition \ref{prop:Tenevaremove}. Thus the approach using the Teneva game gives rise to more crossings than the other one.
\end{prop}
\begin{proof}We let $n$ stand for the number of crossings of the maximal tassel under study, with $n\geq 12$, the other cases being left for the reader. We write $n=6k+k'$ with $k, k'$ positive integers and $0 \leq k' < 6$. We will prove that the difference between the situation where the least number of crossings are created leaning on Proposition \ref{prop:Tenevaremove}, $6\cdot 2^{l_{6k+k'}}-1$, and the situation where more crossings are created leaning on Corollary \ref{cor:count}, $6k+6$, is always non-negative. This amounts to proving that $2^{l_{6k+k'}}-k-1-1/6$ is non-negative. Let
\[
k=2^{e_1} + 2^{e_2} + \dots + 2^{e_{N_k}}, \qquad \qquad \text{ with } \quad e_1 > e_2 > \dots > e_{N_k} \geq 0
\]
be the dyadic expansion of $k$.

The proof will be split into different instances.
\begin{itemize}
\item If $N_k =1$ i.e., $k=2^e$, then
\[
6k+k'=(2^2+2)\cdot 2^e + k'=2^{2+e}+2^{1+e}+k'
\]
and since $2+e-(1+e)=1$ then $l_{6k+k'}=2+e-1=1+e$, according to Proposition \ref{prop:2adicexp}. Then
\[
2^{l_{6k+k'}}-k-1-1/6 = 2^{1+e}-2^e-1-1/6=2^e-1-1/6>0
\]
\item We now assume $N_k \geq 2$, such that $1\leq e_1-e_2\leq 2$ and $\{ e_{N_k}, \dots , e_2, e_1  \} \subsetneq \{ 0, 1, 2, \dots e_1 -2, e_1-1, e_1 \}$

The latter condition implies that
\[
k \leq -1 + \sum_{s=0}^{e_1}2^s  = -1 + \frac{2^{1+e_1}-1}{2-1} = 2^{1+e_1} - 2
\]
On the other hand
\[
6k+k' \geq 6k > 2^2 k = 2^2(2^{e_1} + 2^{e_2} + \dots + 2^{e_{N_k}}) = 2^{2+e_1} + 2^{2+e_2} + \dots + 2^{2+e_{N_k}}
\]
so
\[
1\leq 2+e_1-(2+e_2) = e_1-e_2 \leq 2
\]
thus
\[
l_{6k+k'} \geq (2+e_1)-1 = 1 + e_1
\]
Then
\[
2^{l_{6k+k'}}-k-1-1/6 \geq 2^{1+e_1} - 2^{1+e_1} + 2 -1 -1/6 = 5/6 > 0
\]

\item We now assume $e_1-e_2>2$. We remark that this implies that $\{ e_{N_k}, \dots , e_2, e_1  \} \subsetneq \{ 0, 1, 2, \dots e_1 -2, e_1-1, e_1 \}$.
\[
6k+k' = (2^2+2)(2^{e_1} + 2^{e_2} + \dots + 2^{e_{N_k}}) + k' = (2^{2+e_1} + 2^{e_1}) + (2^{2+e_2} + 2^{e_2})  + \dots + (2^{e_{2+{N_k}}} + 2^{e_{N_k}}) + k'
\]
so
\[
2+e_1-e_1=2 \qquad \qquad \text{ and then } \qquad \qquad l_{6k+k'} = 2+e_1 -1 = 1+e_1
\]
Then
\begin{align*}
2^{l_{6k+k'}}-k-1-1/6 &= 2^{1+e_1}-(2^{e_1} + 2^{e_2} + \dots + 2^{e_{N_k}}) - 1 - 1/6 \geq \\
& \geq 2^{1+e_1} - \sum_{s=0}^{e_1}2^s + 2^{1+e_2} + 2^{2+e_2} -1-1/6  = 2^{1+e_2} + 2^{2+e_2} - 1/6 > 0
\end{align*}

\item Finally we assume that $k=2^{e_1} + 2^{e_1-1} + 2^{e_1-2}+\dots +   2^{2}+  2^{1}+1 = 2^{1+e_1}-1$. Then
\begin{align*}
6k+k'&=(2^2+2)(2^{e_1} + 2^{e_1-1} + 2^{e_1-2}+\dots +   2^{2}+  2^{1}+1)+k' = \\
&=(2^{2+e_1}+2^{1+e_1}) + (2^{1+e_1}+2^{e_1}) +  \dots + (2^{2+1} + 2^{1+1})+(2^{2+0} + 2^{1+0})+k' = \\
&=2^{2+e_1}+(2^{1+e_1} + 2^{1+e_1}) + (2^{e_1} + 2^{e_1}) + \dots +  (2^{1+1}+2^{2+0}) + 2^{1+0}+k' =\\
&=2^{3+e_1}+2^{1+e_1}+\dots +2^3 + 2^1 +k'
\end{align*}
so
\[
3+e_1-(1+e_1) = 2 \qquad \qquad \text{ and then } \qquad \qquad 2^{l_{6k+k'}}=2^{2+e_1}
\]
and so
\[
2^{l_{6k+k'}}-k-1-1/6 = 2^{2+e_1} - 2^{1+e_1} + 1 -1 -1/6 = 2^{1+e_1}  -1/6 > 0
\]
\end{itemize}
This concludes the proof.
\end{proof}

We remark that when choosing the Teneva Game for the delunification process it is perhaps not wise to perform the game to the end. For instance, if the last tassels are $\sigma_1^{\pm 2}$ each of them will contribute with $8$ extra crossings to the delunification process. Had we stopped at the previous step, then the last tassels would have been $\sigma_1^{\pm 6}$ or $\sigma_1^{\pm 5}$ and each of these would have contributed with only $6$ extra crossings to the delunification process. We then propose to use what we call the \emph{Truncated Teneva Game} which we now describe. We set out to perform a regular Teneva Game on the maximal tassel under study but after the performance of a set of Teneva transformations we take the resulting tassels as the final tassels and use Corollary \ref{cor:count} to calculate the number of extra crossings thus obtained. We then ascertain whether it is worth it to carry on the Teneva Game or not. If it is not we stop the Teneva Game here. At the time of writing we do not if it is better to use the Truncated Teneva Game.

\bigbreak

At this point it becomes clear that minimizing the number of crossings is another issue in the delunification process. We now introduce the relevant definitions for dealing with this.

\begin{def.}\label{def:lunefreexingnumber}$[\emph{Lune-free crossing number}]$
Let $L$ be a link. The {\rm Lune-free crossing number of $L$}, notation $LFC\,  (L)$, is the minimum number of crossings needed to assemble a lune-free diagram of $L$.
\end{def.}

\begin{def.}\label{def:plunefreexingnumber}$[\emph{p-Lune-free crossing number}]$
Let $p$ be an odd prime. Let $L$ be a link admitting non-trivial $p$-colorings. The {\rm p-Lune-free crossing number of $L$}, notation $LFC_p\,  (L)$, is the minimum number of crossings needed to assemble a lune-free diagram of $L$ supporting a $p$-minimal coloring.
\end{def.}

\begin{def.}\label{def:pxingnumber}$[\emph{p-crossing number}]$
Let $p$ be an odd prime. Let $L$ be a link admitting non-trivial $p$-colorings. The {\rm p-crossing number of $L$}, notation $C_p\,  (L)$, is the minimum number of crossings needed to assemble a diagram of $L$ supporting a $p$-minimal coloring.
\end{def.}

Figures \ref{fig:3lfctrefoil}, \ref{fig:lffig8} and \ref{fig:min5lffig8} provide illustrative examples for these definitions. Furthermore, in Figure \ref{ff3}, as a result of listing all the lune-free diagrams with $8$ crossings, it is shown that $LFC\, (\text{figure $8$}) = 8$ (cf. Figure \ref{fig:lffig8}).

\begin{figure}[!ht]
    \psfrag{a}{\huge$0$}
    \psfrag{b}{\huge$1$}
    \psfrag{2a-b}{\huge$2a-b$}
    \psfrag{2b-a}{\huge$2$}
    \psfrag{2b-a1}{\Large$2b-a$}
    \psfrag{3b-2a}{\huge$0$}
    \psfrag{3b-2a1}{\Large$3b-2a$}
    \psfrag{4b-3a}{\huge$1$}
    \centerline{\scalebox{.4}{\includegraphics{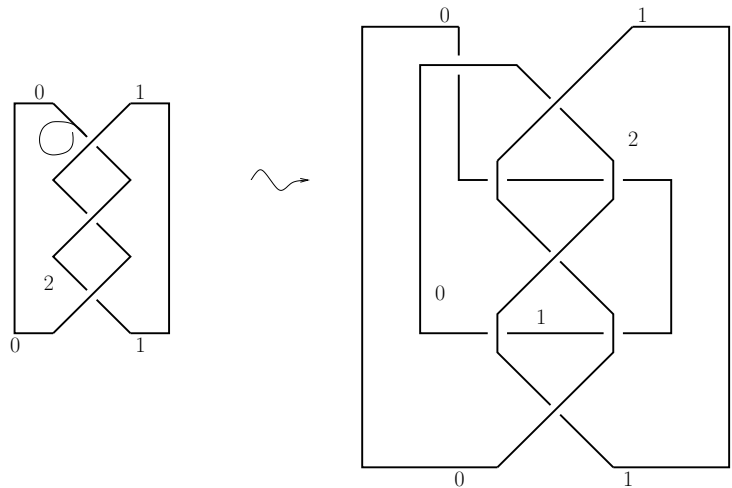}}}
    \caption{Illustrating Definitions \ref{def:lunefreexingnumber}, \ref{def:plunefreexingnumber}, and \ref{def:pxingnumber}: $LFC\, (\text{trefoil}) = 8 = LFC_3\, (\text{trefoil})$, using the result in \cite{EliahouHararyKauffman} stating that no lune-free diagram has less than $8$ crossings. Clearly, $C_3\, (\text{trefoil})=3$}\label{fig:3lfctrefoil}
\end{figure}

\begin{figure}[!ht]
    \psfrag{0}{\huge$0$}
    \psfrag{1}{\huge$1$}
    \psfrag{2a-b}{\huge$2a-b$}
    \psfrag{2}{\huge$2$}
    \psfrag{2b-a1}{\Large$2b-a$}
    \psfrag{3b-2a}{\huge$0$}
    \psfrag{3}{\huge$3$}
    \psfrag{4}{\huge$4$}
    \centerline{\scalebox{.4}{\includegraphics{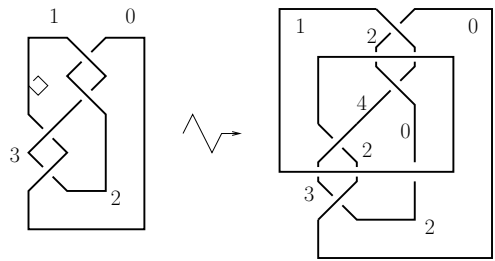}}}
    \caption{Illustrating Definitions \ref{def:lunefreexingnumber}, \ref{def:plunefreexingnumber}, and \ref{def:pxingnumber}: This figure shows that $LFC\, (\text{figure $8$}) \leq 9$. It also illustrates that a lune-free diagram with less crossings may not comply with a non-trivial coloring with less colors. $C_5\, (\text{figure $8$})=4$ as can be seen from the diagram on the left since a crossing number of $3$ has to do with a diagram of the trefoil.}\label{fig:lffig8}
\end{figure}

\begin{figure}[!ht]
    \psfrag{0}{\huge$0$}
    \psfrag{1}{\huge$1$}
    \psfrag{2a-b}{\huge$2a-b$}
    \psfrag{2}{\huge$2$}
    \psfrag{2b-a1}{\Large$2b-a$}
    \psfrag{3b-2a}{\huge$0$}
    \psfrag{3}{\huge$3$}
    \psfrag{4}{\huge$4$}
    \centerline{\scalebox{.4}{\includegraphics{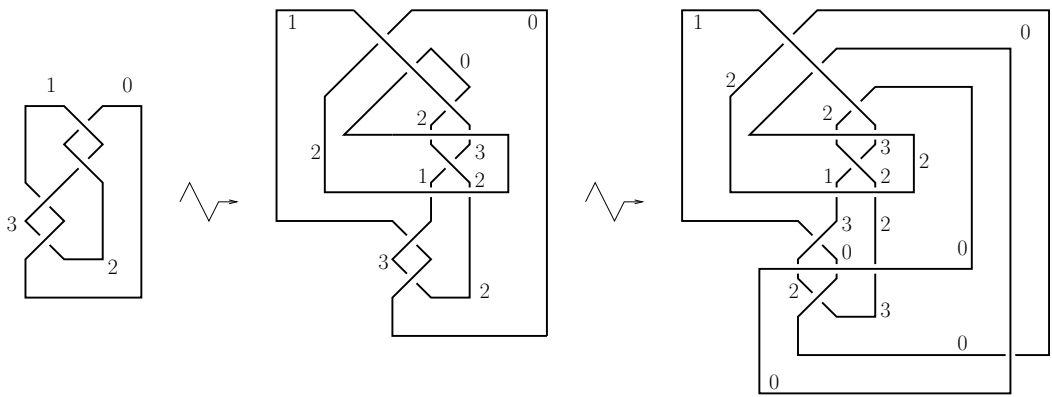}}}
    \caption{Illustrating Definitions \ref{def:lunefreexingnumber}, \ref{def:plunefreexingnumber}, and \ref{def:pxingnumber}: This figure shows that $LFC_5\, (\text{figure $8$}) \leq 10$. We mimic here the technique of Main Lemma \ref{lemma: main} illustrated in Figure \ref{fig:mainl} for the $\sigma_2^{-2}$  and use the rest of the construction for the $\sigma_1^2$.}\label{fig:min5lffig8}
\end{figure}
\bigbreak

\section{Lower bounds on the minimum number of colors: a theory of ``Grey Sets''.}\label{sect:grey}

\noindent

Consider a positive integer $m>2$ and assume $L$ is a link admitting non-trivial $m$-colorings. In order to ascertain if some integer equals $mincol_m\, L$ it is useful to know a lower bound for the minimum number of colors on that modulus. Things are fairly simple for the prime moduli up to $7$ because we know which are the corresponding minimum number of colors. On the other hand, for larger primes, we suspect that the coloring structure is more complex as for instance there being distinct links admitting non-trivial $p$-colorings with distinct minimum numbers of colors. Thus the knowledge of lower bounds on numbers of colors is helpful because the upper bounds are automatically set once we (non-trivially) color a diagram of the link at stake, and in the case of alternating knots of prime determinant, once we know their crossing numbers (\cite{Frank, msolis}).

It is in this set up that we introduce the {\it grey sets}, after recalling the notion of $m$-coloring automorphism (\cite{GJKL, elhamdadi}).
\begin{def.}\label{def:m-coloringauto}$[\emph{m-coloring automorphism} (\cite{GJKL, elhamdadi})]$
Let $m$ be a positive integer.

An $m$-coloring automorphism is a bijection of the integers modulo $m$, $\mathbf{Z}/m\mathbf{Z}$, such that for any $a, b\in \mathbf{Z}$,
\[
f(a\ast b) = f(a)\ast f(b) \quad \text{ mod } m,
\]
where, for any $x, y\in \mathbf{Z}$
\[
x\ast y := 2y-x.
\]

Furthermore, any $m$-coloring automorphism is of the form $f(x)=\mu x + \lambda$ where $\mu$ is a unit from $\mathbf{Z}/m\mathbf{Z}$, and $\lambda$ is any element from $\mathbf{Z}/m\mathbf{Z}$.
\end{def.}
\begin{prop}\label{prop:colauto}Given an integer $m>1$, let $L$ be a link admitting non-trivial $m$-colorings. Let $f$ be an $m$-automorphism. Let $D$ be a diagram of $L$ endowed with a non-trivial $m$-coloring whose distinct colors are $(c_i)_{i=1, \dots , n}$.

Then, $(f(c_i))_{i=1, \dots , n}$ are also the distinct colors of a$($nother$)$ non-trivial $m$-coloring of $D$.

Furthermore, if $m=2k+1$, for some positive integer $k$ then $(c_i)_{i=1, \dots , n}\nsubseteq \{ 0, 1, 2, \dots , k \}$ mod $m$.
\end{prop}
\begin{proof}
The first statement is clear since an $m$ automorphism preserves the $\ast$ operation.

For the proof of the second statement, assume $(c_i)_{i=1, \dots , n}\subset \{ 0, 1, 2, \dots , k \}$ mod $m$. Then the sum of any two colors from this set, or twice any color from this set is strictly less than $2k+1$. This means that any coloring condition satisfied by elements from this set, say $c_{s_i}+c_{s_j}=2c_{s_l}$, is true modulo any integer. This is absurd since as remarked in the introduction we only work with links of non-null determinant.
\end{proof}

The material developing Definition \ref{def:m-coloringauto} into Proposition \ref{prop:colauto} and other results is found in \cite{GJKL}. We are now ready to define \emph{Grey Sets}.

\begin{def.}\label{def:grey}$[\emph{Grey Sets}]$
Let $p(=2k+1)$ be an odd prime.

Let $S$ be a subset of the integers modulo $p$. $S$ is called a \emph{Grey Set} $($or $p$-Grey Set, when there is need to emphasize the modulus$)$ if there is a $p$-coloring automorphism, $f$, such that
\[
f(S) \subseteq \{ 0, 1, 2, \dots , k  \} \qquad \text{ mod } \, p
\]
\end{def.}

That is, a non-trivial $p$-coloring cannot be assembled with the colors from a $p$-grey set alone.
\bigbreak

Lemma \ref{lem:2setsaregrey} proves this is not a vacuous notion.

\begin{lem}\label{lem:2setsaregrey}
Let $p(=2k+1)$ be an odd prime. If $p\geq 3$ $($respect., $p\geq 5)$ then a set with at most two $($respect., at most three$)$ elements modulo $p$ is a $p$-grey set.
\end{lem}
\begin{proof}
The strategy of the proof is to find $p$-coloring automorphisms which alone or composed will map the potential coloring set into $\{ 0, 1, 2, \dots , k \}$. Then Proposition \ref{prop:colauto} implies that this set cannot be a coloring set concluding the proof. We will let $S$ denote the potential coloring set.

We first consider the $p\geq 3$ instance. Let $S = \{  a \}$ be the set described in the statement. Then $f(S)= \{ 0 \}$ with $f(x)=x-a$.

Let $S = \{  a, b \}$ be the subset described in the statement with $0\leq a < b < p$. Then $g(S)= \{ 0, 1 \}$ with 
\[
g(x)=(b-a)^{-1}(x-a) \qquad \text{ mod } \, p
\]

We now let $p\geq 5$ and note that the possibilities for the cardinality of $S$ to be $2$ or $3$ have already been contemplated. Let then $S= \{ a, b, c  \}$ with $0\leq a < b < c < p$.  Then $g(S)= \{ 0, 1, c' \}$ mod $p$ with $1 < c' < p$. If $c'\leq k$ (where $p=2k+1$) the proof is complete. Otherwise, assume $k + 1 \leq c' \leq 2k $. Then $2k+2 \leq 2c' \leq 4k$. If $2k+2 \leq 2c' \leq 3k+1$, then set $h_1(x)=2x$ mod $p$. Then $(h_1\circ g) (S) = \{ 0, 2, c'' \}$ with $1 \leq c'' \leq k$ mod $p$.

Otherwise assume $3k+2 \leq 2c' \leq 4k$. If $k=2l$ (so that $p=2k+1=4l+1$) then $3l+1 \leq c' \leq 4l$. With $h_2(x) = x+l$ mod $p$ we have $(h_2 \circ g) (S) = \{ l, l+1, c'' \}$ with $0 \leq c'' \leq l-1$ mod $p$. Finally, if $k=2l+1$ (so that $p=2k+1=4l+3$) then $3l+3 \leq c' \leq 4l+1$. With $h_3(x)= x + l$ mod $p$ we have $(h_2 \circ g) (S) = \{ l, l+1, c'' \}$ with $0 \leq c'' \leq l-2$ mod $p$. The proof is complete.
\end{proof}

\begin{def.}\label{def:grey index}$[\emph{p-grey Index; rainbow Index}]$
Let $p$ be an odd prime. Let $G_p$ be the largest integer such that any subset of the integers mod $p$  with $G_p$ elements is a $p$-grey set. We call $G_p$ the \emph{grey index of} $p$.

We write
\[
algmincol_p := 1+G_p
\]
and call it the \emph{rainbow index of} $p$.
\end{def.}
\bigbreak

We have done some brute force attempts at calculating the rainbow index for a number of prime moduli with the help of computers; the results are displayed in Table  \ref{Ta:algmincol}.
\begin{table}[h!]
\begin{center}
\scalebox{.9}{\begin{tabular}{| c || c | c | c | c | c |  c |  c | c | c | c | c |  c | c | }\hline
prime \,  $p$ & $3$  & $5$ & $7$ & $11$& $13$ & $17$&  $19$&   $23$&  $29$&  $31$&  $37$&  $41$&  $43$ \\ \hline
$algmincol_p$ & $3$  & $4$ & $4$ & $5$ & $5$  & $6$ &  $6$&    $6$&  $6$&  $6$&  $6$&  $6$&  $6$  \\ \hline\hline
$t_p+2l_p-1$  & - & - & - & $5$ & $6$ & $7$ &  $7$&    $7$&  $8$&  $8$&  $9$&  $9$&  $9$  \\ \hline
\end{tabular}}
\caption{The rainbow index for a number of primes $p$. Note the plateau on $6$ for primes $17$ through $43$. The $t_p+2l_p-1$-line displays upper bounds for the minimum number of colors mod $p$ for $T(2, p)$, rendering the corresponding $algmincol_p$ plausible.}
\label{Ta:algmincol}
\end{center}
\end{table}

Although the plateau on $6$ for the primes $17$ through $43$ seemed strange at first, we realized there were already a number of examples that complied with this data. For example, from Theorem \ref{thm:thenevagame} (see also \cite{klgame}) we obtain the following for torus knots of type $(2, p)$ (note that such a torus knot for prime $p$ has determinant $p$):
\[
mincol_p\, T(2, p) \leq t_p + 2l_p -1
\]
where $l_p$ and $t_p$ are respectively, the length and the tail of the sequence of lower halves for $p$. This estimate complies with the $algmincol_p$ of $6$ for all primes from $17$ through $43$. We show the calculations for $p=17$ and $p=19$ and display the results obtained in the bottom line of Table \ref{Ta:algmincol}.
\[
17 = 2\times 8 + 1 = 2 \times (2\times 4) +1 \quad \qquad LH(17) = (8, 4), \quad l_{17} = 2, \quad t_{17} = 4, \qquad t_{17} + 2l_{17} -1 = 4+4-1 = 7
\]
\[
19 = 2\times 9 + 1 = 2 \times (2\times 4 + 1) +1 \quad \quad LH(17) = (9, 4), \quad l_{19} = 2, \quad t_{19} = 4, \quad t_{19} + 2l_{17} -1 = 4+4-1 = 7
\]

We now stand on firmer grounds in order to look for minimum number of colors for links albeit only up to modulus $43$.

\bigbreak

The following questions seem to be in order at this point. We let $p=2k+1$ stand for an odd prime.
\begin{enumerate}
\item Can we treat the grey index theoretically? That is, are there other ways of calculating/estimating it besides brute force?

\item  How does $G_{2k+1}$ evolve with $k$ ?

\item  Given a modulus $p$, is there a knot (link) $K$ such that $mincol_p\,  K = algmincol_p$ ? If there is a Common $p$-Minimal Sufficient Set of Colors is its cardinality $algmincol_p$ ?
\end{enumerate}
\bigbreak

\section{Algorithms and computational results.}\label{sect:algo}

According to Conway \cite{Conway}, links can be divided into two
basic classes: algebraic and non-algebraic. Algebraic links
(numerator closures of algebraic tangles) can be obtained from
elementary tangles $0$, $1$, and $-1$ using three operations for
the derivation of algebraic tangles: sum, product, and ramification.
A {\it basic polyhedron} \cite{Conway, Kirkman, Kirkman1, Caudron, JablanSazdanovic} (or a {\it
lune-free diagram} \cite{EliahouHararyKauffman}) is a link diagram
without two-sided regions  (bigons). As a graph, it is a 4-valent,
4-edge connected, at least 2-vertex connected graph without bigons.
The main difference between basic polyhedra and the geometrical
polyhedra is that the geometrical polyhedra has to be 3-vertex
connected, and basic polyhedra may be 2-vertex connected. A {bigon
collapse} (or bigon contraction) is the operation that can be used
in order to distinguish algebraic links from non-algebraic ones: after
complete bigon collapse, an algebraic link collapses to a closure of
the tangle $1$, and every minimal (with respect to the number of crossings) diagram of a non-algebraic (or
polyhedral) link collapses into some basic polyhedron.

The first
problem is the derivation of basic polyhedra. This problem was
solved for $n\le 12$ crossings by T.P. Kirkman \cite{Kirkman, Kirkman1}. J.H.
Conway used basic polyhedra for the derivation of knots and links
with $n\le 11$ crossings and for the Conway notation, where polyhedral
links are derived by substituting crossings in basic polyhedra by
algebraic tangles. A. Caudron \cite{Caudron} added the missing
basic polyhedron 12E  to Kirkman's list of basic polyhedra. Hence, the
complete list of basic polyhedra with $n\le 12$ crossings
contains one basic polyhedron $6^*$ with $n=6$ (Borromean rings),
one basic polyhedron $8^*$ (knot $8_{18}$), one basic polyhedron
with $n=9$(knot $9_{40}$), three basic polyhedra $10^*$-$10^{***}$
with $n=10$ (where among them the only knot is $10_{123}=10^*$),
three basic polyhedra $11^*$-$11^{***}$ with $n=11$ and 12 basic
polyhedra 12A-12L with $n=12$ crossings. Derivation of basic
polyhedra with more than $n=12$ crossings became possible thanks to
the use of the computer program  "plantri" written by G. Brinkmann
and B. McKay \cite{BrinkmannMcKay}. In the program {\it LinKnot} \cite{JablanSazdanovic} we
provide the list of basic polyhedra with $n\le 20$ crossings, which
contains 19 basic polyhedra with $n=13$, 64 with $n=14$, 155 with
$n=15$, 510 with$n=16$, 1514 with $n=17$, 5145 with $n=18$, 16966
with $n=19$, and 58782 with $n=20$ crossings.

According to Theorem \ref{thm:main}, every knot $K$ admitting non-trivial $p$-colorings has a $p$-minimal diagram
which is lune-free. Hence, for the exhaustive derivation of such
colorings we implemented the following algorithm:

\bigskip

{\bf Algorithm 1}:
\begin{enumerate}
\item take a basic polyhedron $B$ which is a knot;
\item make all crossing changes in $B$;
\item recognize all knots represented by diagrams obtained in (2) and select among them
diagrams representing $K$;
\item make all $p$-colorings of these diagrams and find the coloring
with the smallest number of colors;
\item apply (1)-(4) to all basic polyhedra (in ascending order), until the first
diagram of $K$  which supports $mincol_p(K)$ is obtained.
\end{enumerate}

In this paper we applied this algorithm to all basic polyhedra with
$n\le 16$ crossings. For all computations we used the program {\it
LinKnot} \cite{JablanSazdanovic}.

This and the next algorithm guarantee the exact computation of the
numbers $LFC_p(K)$ and $C_p(K)$. Notice that by the reduction of the
number of colors ``by hand'' we can  never be sure that we obtained
exact values of these numbers. E.g., knot $7_3$ with $p=13$ and
$mincol_p(7_3)=5$ is represented in \cite{lopesp11} (Fig. 10) with a
lune-free diagram supporting $mincol_p\,(7_3)=5$ with $n=18$
crossings has $LFC_p(7_3)=16$, supported on the lune-free diagram
$16110^*-1.1.1.1.-1.1.-1.-1.-1.1.1.-1.1.-1.1.-1$ (Fig. 1).

\begin{figure}[th]
\centerline{\psfig{file=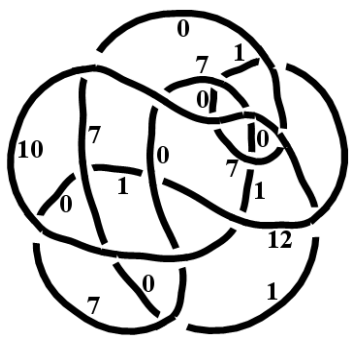,width=1.80in}} \vspace*{8pt}
\caption{Lune-free diagram
$16110^*-1.1.1.1.-1.1.-1.-1.-1.1.1.-1.1.-1.1.-1$ of the knot $7_3$
with $n=16$ crossings. Hence, $LFC_p(7_3)=16$. \label{ff1}}
\end{figure}

The main difficulties for the application of this algorithm are the
enormously large size of the computations (e.g., in order to find
the lune-free diagram for the knot $7_2$ we need to make crossing
changes in all basic polyhedra up to $1595^*$, select 4805 lune-free
diagrams representing $7_2$ and check them for $mincol_p K$), and
the problem of when the algorithm finishes, i.e., step (5). For
$mincol_p K$ we know that $algmincol_p \le mincol_p K$. Hence, if in
step (3) we obtain a number of colors equal to $algmincol_p$
we know that we reached $mincol_p K$. However, we don't know if
for every knot $K$ $mincol_p K = algmincol_p$, i.e., that do not
exist knots for which  $mincol_p K
> algmincol_p$. Therefore, in Table \ref{Ta:Slavik1} we selected only knots
for which we are sure we obtained $mincol_p K$, because
for them  $mincol_p K = algmincol_p$. The other possibility to
confirm $mincol_p K$ on lune-free diagrams is to find any diagram of
$K$ which supports $mincol_p K = algmincol_p$ and which, usually,
has the smaller number of crossings $C_p(K)< LFC_p(K)$ than the
lune-free diagram with the same property. In this case, thanks to
Corollary 2.1 we know that there exists a lune-free diagram of $K$
which supports $mincol_p K = algmincol_p$.

The other algorithm for finding arbitrary diagrams of a knot $K$
supporting $mincol_p K$ is even more complicated, because it works with
all diagrams of $K$, and not just with the lune-free diagrams.

\bigskip

{\bf Algorithm 2}:
\begin{enumerate}
\item take an arbitrary knot diagram $D$;
\item make all crossing changes in $D$;
\item recognize all knots obtained in (2) and select among them
diagrams representing $K$;
\item make all $p$-colorings of these diagrams and find the coloring
with the smallest number of colors;
\item apply (1)-(4) to all diagrams (in ascending order), until the first
diagram of $K$  which supports $mincol_p(K)$ is obtained.
\end{enumerate}

Certainly, because the number of different diagrams given by
crossing changes of the knot $K$ is enormously large, this algorithm is
almost impossible for the practical application, especially because
even small changes in diagrams representing $K$ can result in
different number of colors. E.g., the knot $K=4_1=2\,2$ has
$C_5(4_1)=4$, because its  $mincol_5(4_1)=4$ is supported on its
minimal diagram. Let's consider two non-minimal diagrams of $K$,
$D_1=(-1,1,1,1)\,2$ and $D_2=(1,-1,1,1)\,2$, which differ one from
the other only in that two first crossings changed their places. The
number of colors necessary for coloring $D_1$ is 5, and for $D_2$ is
4 (Figure \ref{ff2}).

\begin{figure}[th]
\centerline{\psfig{file=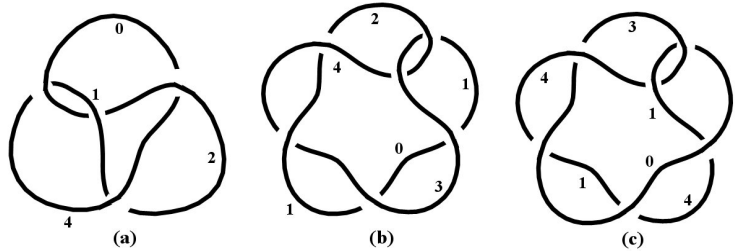,width=3.80in}} \vspace*{8pt}
\caption{(a) Minimal diagram of the knot $4_1=2\,2$ supporting
$mincol_5(4_1)=4$; (b) non-minimal diagram of the same knot colored
with 5 colors; (c) non-minimal diagram of the same knot colored with
4 colors. \label{ff2}}
\end{figure}
\bigskip

{\bf Question:} Is the knot $4_1$ the only knot supporting
$mincol_p(K)$ ($p\ge 5$) on its minimal diagram?







\begin{table}[h!]
\begin{center}
\scalebox{.86}{\begin{tabular}{|c|c|c|c|c|c|} \hline

$K$ & $Con$  & $Lune-free$ $diagram$  & $p$ & $mincol_p(K)$ & $LFC_p(K)$  \\
\hline\hline

$4_1$ & $2\,2$  & $111^*-1.1.-1.-1.1.1.1.-1.-1.1.1$  & 5 & 4 & 11  \\
\hline

$5_1$ &  $5$ & $101^*-1.-1.-1.1.-1.1.-1.1.1.1$  &  5 & 4  & 10  \\
\hline

$5_2$ & $3\,2$   & $8^*-1.-1.-1.1.-1.1.1.1$   &  7  & 4 & 8 \\
\hline

$6_2$ &  $3\,1\,2$ & $9^*1.1.1.1.1.1.1.-1.-1$  & 11 &  5  & 9 \\
\hline

$6_3$ & $2\,1\,1\,2$  & $122^*-1.-1.-1.1.-1.1.1.1.1.-1.1.-1$  & 13 &
5 & 12 \\ \hline

$7_1$ &  $7$ & $1420^*1.1.-1.1.1.-1.-1.1.-1.1.-1.-1.1.-1$  &  7 & 4
& 14
\\ \hline

$7_2$ &  $5\,2$ &   & 11 &  5 & \\ \hline

$7_3$ & $4\,3$  & $16110^*-1.1.1.1.-1.1.-1.-1.-1.1.1.-1.1.-1.1.-1$ &
13 &  5 & 16 \\ \hline

$7_5$ &  $3\,2\,2$ & $122^*-1.1.1.1.-1.-1.-1.1.1.-1.1.-1$  & 17 & 6
& 12
\\ \hline

$7_6$ & $2\,2\,1\,2$  & $148^*-1.-1.1.-1.1.1.1.1.1.-1.-1.-1.1.-1$  &
19 & 6 & 14 \\ \hline

$8_2$ &  $5\,1\,2$ &   & 17 &  6 & \\ \hline

$9_{42}$ & $2\,2,3,-2$   &
$1420^*-1.-1.1.1.1.-1.1.-1.1.1.-1.-1.-1.1$ & 7  & 4 & 14 \\ \hline

$10_{128}$ & $3\,2,3,-2$   &   &  5 & 4 & \\ \hline

$10_{125}$ & $5,2\,1,-2$   &   &  11 & 5 & \\ \hline

$10_{132}$ & $2\,3,3,-2$   & $138^*1.-1.1.1.-1.-1.-1.1.1.-1.1.1.1$ &
5 & 4 & 13 \\ \hline

$10_{152}$ & $(3,2)\,-(3,2)$   &   &  11 & 5 & \\ \hline

$10_{154}$ &   $(2\,1,2)\,-(2\,1,2)$ &
  & 13  & 5 & \\ \hline

$10_{161}$ & $3:2\,0:-2\,0$  &
$1315^*-1.-1.1.1.1.1.-1.1.-1.1.1.1.-1$  &  5 & 4 & 13 \\ \hline

$11n19$ & $5,2\,2,-2$  &  $1211^*-1.-1.-1.1.1.1.-1.-1.1.1.1.-1$ &  5
& 4 & 12 \\ \hline

$11n135$ & $2\,2:-2\,0:-2\,0$  &
$1452^*1.1.1.-1.1.-1.1.-1.-1.-1.1.1.-1.1$  & 5 & 4 & 14 \\ \hline

\end{tabular}}
\caption{This is the list of knots with $mincol_p(K)$ supported on
the lune-free diagrams, obtained by using Algorithm 1. Every knot is
given by its classical symbol, Conway symbol, $p$, the lune-free
diagram supporting $mincol_p(K)$, and with $mincol_p(K)$. The knots
where the lune-free diagram supporting $LFC_p(K)$ is omitted have
such diagram with more than $n=16$ crossings.}
\label{Ta:Slavik1}
\end{center}
\end{table}
\normalsize

\bigskip

\begin{figure}[th]
\centerline{\psfig{file=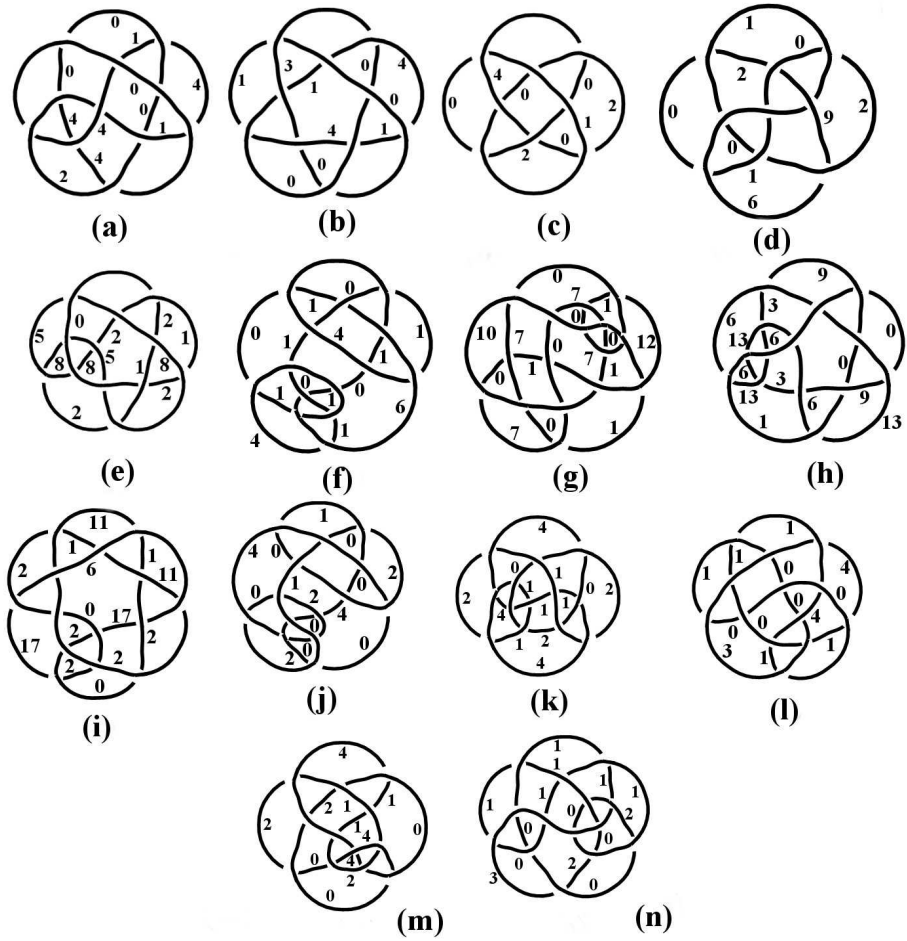,width=4.80in}} \vspace*{8pt}
\caption{Lune-free diagram of the knots (a) $4_1$; (b) $5_1$; (c)
$5_2$; (d) $6_2$; (e) $6_3$; (f) $7_1$; (g) $7_3$; (h) $7_5$; (i)
$7_6$; (j) $9_{42}$; (k) $10_{132}$; (l) $10_{161}$; (m) $11n19$;
(n) $11n135$ with $LFC_p(K)$ supported on them. \label{ff4}}
\end{figure}


\bigskip


\begin{figure}[th]
\centerline{\psfig{file=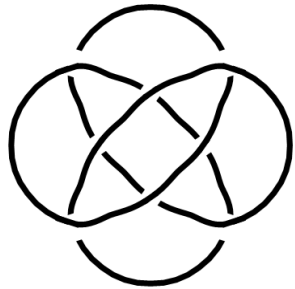,width=1.50in}} \vspace*{8pt}
\caption{Lune-free diagram of the knot $4_1$ with $n=8$ crossings.
Hence, $LFC(4_1)=8$ (compare with Fig. 9). \label{ff3}}
\end{figure}



\begin{table}[h!]
\begin{center}
\scalebox{.93}{\begin{tabular}{| c | c | c ||| c | c | c |||  c |  c | c ||| c | c | c | }\hline
$K$ & $Con$ & $LFC$ &     $K$ & $Con$ & $LFC$ &      $K$ & $Con$ & $LFC$ &     $K$ & $Con$ & $LFC$ \\ \hline \hline
$3_{1}$ & $3$ & $8$ &           $8_{8}$ & $2\,3\,1\,2$ & $11$ &       $9_{8}$ & $2\,4\,1\,2$ & $14$ &           $9_{29}$ & $.2.2\,0.2$ & $14$        \\ \hline
$4_{1}$ & $2\,2$ & $8$ &        $8_{9}$ & $3\,1\,1\,3$ & $10$ &       $9_{9}$ & $4\,2\,3$ & $14$ &              $9_{30}$ & $2\,1\,1,2\,1,2$ & $13$     \\ \hline
$5_{1}$ & $5$ & $8$ &           $8_{10}$ & $2\,1,3,2$ & $13$ &        $9_{10}$ & $3\,3\,3$ &$14$ &              $9_{31}$ & $2\,1\,1\,1\,1\,1\,2$ & $11$ \\ \hline
$5_{2}$ & $3\,2$ & $8$ &        $8_{11}$ & $3\,2\,1\,2$ & $12$   &    $9_{11}$ & $4\,1\,2\,2$ & $14$ &          $9_{32}$ & $.2\,1.2\,0$ & $11$   \\ \hline
$6_{1}$ & $4\,2$ & $9$ &        $8_{12}$ & $2\,2\,2\,2$ &  $12$ &     $9_{12}$ & $4\,2\,1\,2$ & $14$ &          $9_{33}$ & $.2\,1.2$ & $11$   \\ \hline
$6_{2}$ & $3\,1\,2$ & $9$   &   $8_{13}$ & $3\,1\,1\,1\,2$ & $11$ &   $9_{13}$ & $3\,2\,1\,3$ & $14$ &          $9_{34}$ & $8^*2\,0$ & $11$  \\ \hline
$6_{3}$ & $2\,1\,1\,2$ & $8$ &  $8_{14}$ &  $2\,2\,1\,1\,2$ &  $12$ & $9_{14}$ & $4\,1\,1\,1\,2$ & $12$ &       $9_{35}$ & $3,3,3$ & $15$  \\ \hline
$7_{1}$ & $7$ & $11$ &          $8_{15}$ & $2\,1,2\,1,2$ & $12$ &     $9_{15}$ & $2\,3\,2\,2$ & $14$ &          $9_{36}$ & $2\,2,3,2$ & $14$  \\ \hline
$7_{2}$ & $5\,2$ & $11$ &       $8_{16}$ & $.2.2\,0$ & $11$   &       $9_{16}$ & $3,3,2+$ & $14$ &              $9_{37}$ & $2\,1,2\,1,3$ & $14$   \\ \hline
$7_{3}$ & $4\,3$ & $11$ &       $8_{17}$ & $.2.2$ & $10$ &            $9_{17}$ & $2\,1\,3\,1\,2$ & $12$ &       $9_{38}$ & $.2.2.2$ & $14$ \\ \hline
$7_{4}$ & $3\,1\,3$ & $11$ &    $8_{18}$ & $8^*$ & $8$ &              $9_{18}$ & $3\,2\,2\,2$ & $14$ &          $9_{39}$ & $2:2:2\,0$ & $12$  \\ \hline
$7_{5}$ & $3\,2\,2$ & $11$ &    $8_{19}$ & $3,3,-2$ & $9$ &           $9_{19}$ & $2\,3\,1\,1\,2$ & $12$ &       $9_{40}$ & $9^*$ & $9$  \\ \hline
$7_{6}$ & $2\,2\,1\,2$ & $11$ & $8_{20}$ & $3,2\,1,-2$ & $10$ &       $9_{20}$ & $3\,1\,2\,1\,2$ & $13$ &       $9_{41}$ & $2\,0:2\,0:2\,0$ & $12$   \\ \hline
$7_{7}$ & $2\,1\,1\,1\,2$ & $9$ & $8_{21}$ &  $2\,1,2\,1,-2$ & $9$  & $9_{21}$ & $3\,1\,1\,2\,2$ & $14$ &       $9_{42}$ & $2\,2,3,-2$ & $11$  \\ \hline
$8_{1}$ & $6\,2$ & $12$&        $9_{1}$ & $9$ & $14$ &                $9_{22}$ & $2\,1\,1,3,2$ & $12$ &         $9_{43}$ & $2\,1\,1,3,-2$ & $11$   \\ \hline
$8_{2}$ & $5\,1\,2$ & $12$ &    $9_{2}$ & $7\, 2$ & $14$ &            $9_{23}$ & $2\,2\,1\,2\,2$ & $14$ &       $9_{44}$ & $2\,2,2\,1,-2$ & $11$ \\ \hline
$8_{3}$ & $4\,4$ & $12$ &       $9_{3}$ & $6\,3$ & $14$ &             $9_{24}$ & $2\,1,3,2+$ & $13$ &           $9_{45}$ & $2\,1\,1,2\,1,-2$ & $11$  \\ \hline
$8_{4}$ & $4\,1\,3$ & $12$ &    $9_{4}$ & $5\,4$ & $14$ &             $9_{25}$ & $2\,2,2\,1,2$ & $14$ &         $9_{46}$ & $3,3,-3$ & $11$  \\ \hline
$8_{5}$ & $3,3,2$ & $13$ &      $9_{5}$ & $5\,1\,3$ & $14$ &          $9_{26}$ & $3\,1\,1\,1\,1\,2$ & $11$ &    $9_{47}$ & $8^*-2\,0$ & $11$  \\ \hline
$8_{6}$ & $3\,3\,2$ & $12$  &   $9_{6}$ & $5\,2\,2$ & $14$ &          $9_{27}$ & $2\,1\,2\,1\,1\,2$ & $13$ &    $9_{48}$ & $2\,1,2\,1,-3$ & $12$  \\ \hline
$8_{7}$ & $4\,1\,1\,2$ & $11$ & $9_{7}$ & $3\,4\,2$ & $14$ &          $9_{28}$ & $2\,1,2\,1,2+$ & $13$ &        $9_{49}$ & $-2\,0:-2\,0:-2\,0$ & $12$  \\ \hline
\end{tabular}}
\caption{In this table and the next two, we present $LFC(K)$ for each knot $K$ from the Rolfsen tables
\cite{Rolfsen}  (in the classical notation and in Conway's). From the Conway symbol or its DT code is easy
to make drawings of the corresponding lune-free diagrams.}
\label{Ta:Slavik2}
\end{center}
\end{table}

\begin{table}[h!]
\begin{center}
\scalebox{.9}{\begin{tabular}{| c | c | c ||| c | c | c |||  c |  c | c ||| c | c | c | }\hline
$K$ & $Con$ & $LFC$ &     $K$ & $Con$ & $LFC$ &      $K$ & $Con$ & $LFC$ &     $K$ & $Con$ & $LFC$ \\ \hline \hline
$10_{1}$ & $8\,2$ & $15$  &    $10_{22}$ & $3\,3\,1\,3$ & $15$ &   $10_{43}$ & $2\,1\,2\,2\,1\,2$ & $14$ &      $10_{64}$ & $3\,1,3,3$ & $15$     \\ \hline
$10_{2}$ & $7\,1\,2$ & $15$ &  $10_{23}$ & $3\,3\,1\,1\,2$ &$14$ &  $10_{44}$ & $2\,1\,2\,1\,1\,1\,2$ & $14$&   $10_{65}$ & $3\,1,2\,1,3$ & $14$ \\ \hline
$10_{3}$ & $6\,4$ & $15$ &$10_{24}$ & $3\,2\,3\,2$ & $15$ &  $10_{45}$ & $2\,1\,1\,1\,1\,1\,1\,2$ & $12$  &  $10_{66}$ & $3\,1,2\,1,2\,1$ & $15$\\ \hline
$10_{4}$ & $6\,1\,3$ & $15$ &  $10_{25}$ & $3\,2\,2\,1\,2$& $15$ & $10_{46}$ & $5,3,2$ & $16$ &                 $10_{67}$ & $2\,2,2\,1,3$ & $15$ \\ \hline
$10_{5}$ & $6\,1\,1\,2$ &$14$& $10_{26}$ & $3\,2\,1\,1\,3$ & $15$ &  $10_{47}$ & $2\,1,5,2$ & $15$ &           $10_{68}$ & $2\,1\,1,3,3$ & $15$\\ \hline
$10_{6}$ & $5\,3\,2$ & $15$  & $10_{27}$ &$3\,2\,1\,1\,1\,2$ & $14$ &  $10_{48}$ & $4\,1,3,2$ & $15$&      $10_{69}$ & $2\,1\,1,2\,1,2\,1$ & $14$ \\ \hline
$10_{7}$ & $5\,2\,1\,2$ &$15$& $10_{28}$ & $3\,1\,3\,1\,2$ & $14$ & $10_{49}$ & $4\,1,2\,1,2$ & $15$ &      $10_{70}$ & $2\,2,3,2+$ & $15$  \\ \hline
$10_{8}$ & $5\,1\,4$ & $15$ &  $10_{29}$ & $3\,1\,2\,2\,2$ & $15$ &    $10_{50}$ & $3\,2,3,2$ & $15$ &      $10_{71}$ & $2\,2,2\,1,2+$ & $14$ \\ \hline
$10_{9}$ & $5\,1\,1\,3$ & $15$& $10_{30}$ & $3\,1\,2\,1\,1\,2$& $15$ & $10_{51}$ & $3\,2,2\,1,2$ & $14$ &  $10_{72}$ & $2\,1\,1,3,2+$ & $14$ \\ \hline
$10_{10}$&$5\,1\,1\,1\,2$&$14$& $10_{31}$ & $3\,1\,1\,3\,2$ & $14$ &   $10_{52}$ & $3\,1\,1,3,2$   & $15$ &   $10_{73}$ & $2\,1\,1,2\,1,2+$ & $14$\\ \hline
$10_{11}$ & $4\,3\,3$ & $15$ & $10_{32}$ & $3\,1\,1\,1\,2\,2$ & $15$ & $10_{53}$ & $3\,1\,1,2\,1,2$ & $15$ & $10_{74}$ & $2\,1,3,3+$ & $15$ \\ \hline
$10_{12}$ & $4\,3\,1\,2$ & $14$&$10_{33}$ & $3\,1\,1\,1\,1\,3$ & $12$ & $10_{54}$ &      $2\,3,3,2$ & $15$ &   $10_{75}$ & $2\,1,2\,1,2\,1+$ & $14$   \\ \hline
$10_{13}$ & $4\,2\,2\,2$ &$15$ &  $10_{34}$ & $2\,5\,1\,2$ & $14$ &  $10_{55}$ & $2\,3,2\,1,2$ & $15$ &       $10_{76}$ & $3,3,2++$ & $16$    \\ \hline
$10_{14}$ & $4\,2\,1\,1\,2$&$15$& $10_{35}$ & $2\,4\,2\,2$ & $15$ &  $10_{56}$ & $2\,2\,1,3,2$ & $15$   & $10_{77}$ & $2\,1,3,2++$ & $15$ \\ \hline
$10_{15}$ & $4\,1\,3\,2$ & $14$ & $10_{36}$ & $2\,4\,1\,1\,2$ & $15$ &  $10_{57}$ & $2\,2\,1,2\,1,2$ & $14$ &  $10_{78}$ & $2\,1,2\,1,2++$ & $15$ \\ \hline
$10_{16}$ & $4\,1\,2\,3$ & $15$ & $10_{37}$ & $2\,3\,3\,2$ & $14$ &  $10_{58}$ & $2\,2,2\,2,2$ & $15$ &   $10_{79}$ & $(3,2)\,(3,2)$ & $15$\\ \hline
$10_{17}$ & $4\,1\,1\,4$ & $14$ & $10_{38}$ & $2\,3\,1\,2\,2$ & $15$ & $10_{59}$ & $2\,1\,1,2\,2,2$ & $15$ & $10_{80}$ & $(3,2)\,(2\,1,2)$ & $15$  \\ \hline
$10_{18}$ & $4\,1\,1\,2\,2$ & $15$ &$10_{39}$ & $2\,2\,3\,1\,2$ & $15$ & $10_{60}$& $2\,1\,1,2\,1\,1,2$ & $14$ & $10_{81}$ & $(2\,1,2)\,(2\,1,2)$ & $14$\\ \hline
$10_{19}$ & $4\,1\,1\,1\,3$ & $13$ & $10_{40}$ &$2\,2\,2\,1\,1\,2$ & $14$ & $10_{61}$ & $4,3,3$ & $16$ &        $10_{82}$ & $.4.2$ & $13$ \\ \hline
$10_{20}$ & $3\,5\,2$ & $15$ &  $10_{41}$ & $2\,2\,1\,2\,1\,2$ & $14$ &  $10_{62}$ & $2\,1,4,3$ & $15$ &      $10_{83}$ & $.3\,1.2$ & $13$  \\ \hline
$10_{21}$ & $3\,4\,1\,2$ & $15$&  $10_{42}$ & $2\,2\,1\,1\,1\,1\,2$ & $14$ & $10_{63}$ & $2\,1,2\,1,4$ & $15$&  $10_{84}$  & $.2\,2.2$ & $13$  \\ \hline
\end{tabular}}
\caption{$LFC(K)$ for each knot $K$ from the Rolfsen tables \cite{Rolfsen} (cont'd).}
\label{Ta:Slavik2a}
\end{center}
\end{table}

\begin{table}[h!]
\begin{center}
\scalebox{.8}{\begin{tabular}{| c | c | c ||| c | c | c |||  c |  c | c ||| c | c | c | }\hline
$K$ & $Con$ & $LFC$ &     $K$ & $Con$ & $LFC$ &      $K$ & $Con$ & $LFC$ &     $K$ & $Con$ & $LFC$ \\ \hline \hline
$10_{85}$ & $.4.2\,0$ & $13$ & $10_{106}$ & $3\,0:2:2\,0$ & $15$ & $10_{127}$ & $4\,1,2\,1,-2$ & $13$ &$10_{148}$ & $(3,2)\,(3,-2)$ & $13$ \\ \hline
$10_{86}$ & $.3\,1.2\,0$ & $12$ & $10_{107}$ & $2\,1\,0:2:2\,0$ & $14$ & $10_{128}$ & $3\,2,3,-2$ & $12$ &  $10_{149}$ & $(3,2)\,(2\,1,-2)$ & $12$  \\ \hline
$10_{87}$ & $.2\,2.2\,0$ & $13$ & $10_{108}$ & $3\,0:2\,0:2 0$ & $13$ & $10_{129}$ & $3\,2,2\,1,-2$ & $11$ &$10_{150}$ & $(2\, 1,2)\,(3-2)$ & $12$   \\ \hline
$10_{88}$ & $.2\,1.2\,1$ & $12$ & $10_{109}$ & $2.2.2.2$ & $14$ &  $10_{130}$ & $3\,1\,1,3,-2$ & $13$ & $10_{151}$ & $(2\,1,2)\,(2\,1,-2)$ & $13$   \\ \hline
$10_{89}$ & $.2\,1.2\,1\,0$ & $12$ & $10_{110}$ & $2.2.2.2\,0$ & $15$ & $10_{131}$ & $3\,1\,1,2\,1,-2$ & $12$&   $10_{152}$ & $(3,2)\,-(3,2)$ & $14$  \\ \hline
$10_{90}$ & $.3.2.2$ & $15$ & $10_{111}$ & $2.2.2\,0.2$ & $15$ & $10_{132}$ & $2\,3,3,-2$ & $11$ & $10_{153}$ & $(3,2)\,-(2\,1,2)$ & $13$  \\ \hline
$10_{91}$ & $.3.2.2\,0$ & $14$ & $10_{112}$ & $8^*3$ & $13$ &  $10_{133}$ & $2\,3,2\,1,-2$ & $12$ & $10_{154}$ & $(2\,1,2)\,-(2\,1,2)$ & $12$   \\ \hline
$10_{92}$ & $.2\,1.2.2\,0$ & $14$ &  $10_{113}$ & $8^*2\,1$ & $12$ & $10_{134}$ & $2\,2\,1,3,-2$ & $12$ & $10_{155}$ & $-3:2:2$ & $13$  \\ \hline
$10_{93}$ & $.3.2\,0.2$ & $15$ & $10_{114}$ & $8^*3\,0$ & $13$ &  $10_{135}$ & $2\,2\,1,2\,1,-2$ & $13$ & $10_{156}$ & $-3:2:2 0$ & $11$  \\ \hline
$10_{94}$ & $.3\,0.2.2$ & $15$    &  $10_{115}$ & $8^*2\,0.2\,0$ & $12$ & $10_{136}$ & $2\,2,2\,2,-2$ & $11$ & $10_{157}$ & $-3:2\,0:2\,0$ & $12$ \\ \hline
$10_{95}$ & $.2\,1\,0.2.2$ & $14$ & $10_{116}$ & $8^*2:2$ & $13$ &   $10_{137}$ & $2\,2,2\,1\,1,-2$ & $12$ &   $10_{158}$ & $-3\,0:2:2$ & $13$  \\ \hline
$10_{96}$ & $.2.2\,1.2$ & $15$ & $10_{117}$ & $8^*2:2\,0$ & $12$ &  $10_{138}$ & $2\,1\,1,2\,1\,1,-2$ & $12$ &$10_{159}$ & $-3\,0:2:2\,0$ & $11$  \\ \hline
$10_{97}$ & $.2.2\,1\,0.2$ & $15$ &  $10_{118}$ & $8^*2:.2$ & $12$ &  $10_{139}$ & $4,3,-2\,-1$ & $13$ & $10_{160}$ & $-3\,0:2\,0:2\,0$ & $12$  \\ \hline
$10_{98}$ & $.2.2.2.2\,0$  & $15$ & $10_{119}$ & $8^*2:.2\,0$ & $13$ & $10_{140}$ & $4,3,-3$ & $11$ &  $10_{161}$ & $3:-2\,0:-2\,0$ & $12$  \\ \hline
$10_{99}$ & $.2.2.2\,0.2\,0$ & $15$ & $10_{120}$ & $8^*2\,0::2\,0$ & $15$ & $10_{141}$ & $4,2\,1,-3$ & $11$ &$10_{162}$&$-3\,0:-2\,0:-2\,0$ & 13  \\ \hline
$10_{100}$ & $3:2:2$  & $15$ & $10_{121}$ & $9^*2\,0$ & $12$ &  $10_{142}$ & $3\,1,3,-2\,-1$ & $12$ & $10_{163}$&$8^*-3\,0$ & 11 \\ \hline
$10_{101}$ & $2\,1:2:2$ & $15$ & $10_{122}$ & $9^*.2\,0$ & $13$ &  $10_{143}$ & $3\,1,3,-3$ & $11$ & $10_{164}$&$8^*2:-2\,0$& 11   \\ \hline
$10_{102}$ & $3:2:2\,0$ & $15$ & $10_{123}$ & $10^*$ & $10$ & $10_{144}$ & $3\,1,2\,1,-3$ & $12$ &  $10_{165}$&$8^*2:.-2\,0$& $12$ \\ \hline
$10_{103}$ & $3\,0:2:2$ & $14$  & $10_{124}$ & $5,3,-2$ & $11$ &  $10_{145}$ & $2\,2,3,-2\,-1$ & $12$ &   - & - & - \\ \hline
$10_{104}$ & $3:2\,0:2\,0$ & $13$ & $10_{125}$ & $5,2\,1,-2$ & $13$ & $10_{146}$ & $2\,2,2\,1,-3$ & $11$ & - & - & -  \\ \hline
$10_{105}$ & $2\,1:2\,0:2\,0$ & $14$ & $10_{126}$ & $4\,1,3,-2$ & $13$ & $10_{147}$ & $2\,1\,1,3,-3$ & $12$ &  -  & - & - \\ \hline
\end{tabular}}
\caption{$LFC(K)$ for each knot $K$ from the Rolfsen tables \cite{Rolfsen} (concl).}
\label{Ta:Slavik2b}
\end{center}
\end{table}

\section{Acknowledgements}\label{sect:orgackn}

\noindent

S.J. thanks for support through project no. 174012 financed by the Serbian Ministry of Education, Science and Technological Development.

P.L. acknowledges support from FCT (Funda\c c\~ao para a Ci\^encia e a Tecnologia), Portugal, through project FCT EXCL/MAT-GEO/0222/2012, ``Geometry and Mathematical Physics''.

\end{document}